\documentclass[10pt]{amsart}
\usepackage{graphicx} 

\calclayout

\title{witt affine springer theory}
\author{Noam Nissan}
\address{Weizmann Institute of Science, 234 Herzl Street, Rehovot 7610001, Israel.}
\email{noam.zimhoni@weizmann.ac.il}

\author{Yakov Varshavsky}
\address{Einstein Institute of Mathematics, The Hebrew University of Jerusalem,
Givat Ram, Jerusalem, 9190401, Israel.}
\email{yakov.varshavsky@mail.huji.ac.il}

\usepackage{quiver}
\usepackage{hyperref}
\usepackage[shortlabels]{enumitem}
\usepackage{biblatex}

\usepackage{amsthm}
\usepackage{amsmath}
\usepackage{twemojis}
\usepackage[T1]{fontenc}

\usepackage[nameinlink, capitalise,noabbrev]{cleveref}

\newtheorem*{lemma*}{Lemma}
\newtheorem*{definition*}{Definition}
\newtheorem*{problem*}{Problem}
\newtheorem*{corollary*}{Corollary}

\newtheorem*{proposition*}{Proposition}

\newtheorem*{theorem*}{Theorem}
\newtheorem{remark*}{Remark}

\newtheorem{theorem}{Theorem}[section]
\newtheorem{lemma}[theorem]{Lemma}

\newtheorem{corollary}[theorem]{Corollary}
\newtheorem{proposition}[theorem]{Proposition}

\theoremstyle{definition}
\newtheorem{remark}[theorem]{Remark}

\newtheorem{example}[theorem]{Example}
\newtheorem{para}[theorem]{}
\newtheorem{definition}[theorem]{Definition}

\newlist{mylist}{enumerate}{2}
\setlist[mylist, 1]{label=(\alph*),noitemsep,topsep=0pt, ,ref={\thetheorem}(\alph*)}
\setlist[mylist, 2]{label=\roman*.}

\crefname{Appendix}{Appendix}{Appendixes}
\crefname{mylisti}{\S}{\S}
\crefname{paragraph}{\S}{\S}
\AddToHook{env/lemma/begin}{\crefalias{theorem}{lemma}}
\AddToHook{env/proposition/begin}{\crefalias{theorem}{proposition}}
\AddToHook{env/remark/begin}{\crefalias{theorem}{remark}}
\AddToHook{env/para/begin}{\crefalias{theorem}{paragraph}}
\AddToHook{env/example/begin}{\crefalias{theorem}{example}}
\AddToHook{env/definition/begin}{\crefalias{theorem}{definition}}
\AddToHook{env/corollary/begin}{\crefalias{theorem}{corollary}}

\newcommand{\rats}{\mathbb{Q}}
\newcommand{\ints}{\mathbb{Z}}
\newcommand{\nats}{\mathbb{N}}
\newcommand{\basicextension}{K}
\newcommand{\resfield}{\kappa}
\newcommand{\witt}{{W_{\basicints}}}
\newcommand{\wittzp}{{W_{\ints_p}}}
\newcommand{\wittn}{W_{\basicints,n}}
\DeclareMathOperator{\weilres}{Res}
\newcommand{\basicints}{\mathcal{O}}
\newcommand{\valuation}{v_{\uniformizer}}
\newcommand{\qlbar}{\overline{\rats_l}}

\DeclareMathOperator{\colim}{colim}
\DeclareMathOperator{\op}{op}
\DeclareMathOperator{\Fun}{Fun}

\DeclareMathOperator{\End}{End}
\DeclareMathOperator{\Id}{Id}

\DeclareMathOperator{\mor}{Mor}
\DeclareMathOperator{\ob}{Ob}
\newcommand{\varcat}[1]{\mathcal{#1}}
\DeclareMathOperator{\pro}{Pro}

\newcommand{\dimfunc}{\underline{\dim}}

\DeclareMathOperator{\Alg}{Alg}

\DeclareMathOperator{\perf}{perf}

\DeclareMathOperator{\pfp}{pfp}

\DeclareMathOperator{\fp}{fp}

\DeclareMathOperator{\qcqs}{qcqs}
\DeclareMathOperator{\sch}{Sch}
\newcommand{\persch}{\sch_{\perf}}
\DeclareMathOperator{\aff}{Aff}
\newcommand{\peraff}{\aff_{\perf}}
\DeclareMathOperator{\stacks}{St}
\newcommand{\perstacks}{\stacks_{\perf}}
\DeclareMathOperator{\algsp}{AlgSp}
\newcommand{\peralgsp}{\algsp_{\perf}}
\DeclareMathOperator{\prestacks}{preSt}

\DeclareMathOperator{\Shv}{Shv}

\newcommand{\varstack}[1]{\mathcal{#1}}

\DeclareMathOperator{\catofsets}{Set}
\DeclareMathOperator{\catofanima}{Anima}

\newcommand\liealg[1][g]{\mathfrak{#1}}

\DeclareMathOperator{\Ind}{Ind}
\newcommand{\thegroup}{G}
\newcommand{\theborel}{B}
\newcommand{\thetorus}{T}
\newcommand{\theweylgroup}{W}
\newcommand{\theextendedweylgroup}{\widetilde{\theweylgroup}}
\newcommand{\theroots}{R}
\newcommand{\thetoruslie}{\liealg[t]}
\newcommand{\theborellie}{\liealg[b]}
\newcommand{\thelie}{\liealg[g]}

\newcommand{\chevalley}{\liealg[c]}
\newcommand{\thechevalleymap}{\chi}
\newcommand{\theiwahorychevalleymap}{\nu}
\DeclareMathOperator{\theadjoint}{Ad}

\newcommand{\thediscriminant}{\mathfrak{D}}
\DeclareMathOperator{\regsemisimple}{rs}

\newcommand{\rootfunc}{\textbf{r}}

\newcommand{\theliecenter}{\liealg[z]}

\newcommand{\theiwahori}{I}

\newcommand{\flag}{\operatorname{Fl}}

\newcommand{\arcspace}{\mathcal{L}^{+}}

\newcommand{\loopspace}{\mathcal{L}}
\DeclareMathOperator{\topnilp}{tn}

\DeclareMathOperator{\differential}{D}
\DeclareMathOperator{\codim}{codim}

\DeclareMathOperator{\spec}{Spec}
\newcommand{\uniformizer}{\varpi}
\newcommand{\affinespace}{\mathbb{A}}
\newcommand{\projectivespace}{\mathbb{P}}
\newcommand{\multiplicativegroup}{\mathbb{G}_m}

\newcommand{\explicitset}[1]{\left\{#1\right\}}

\newcommand{\iwahory}{I}
\newcommand{\liewahory}{\operatorname{Lie}(\theiwahori)}
\newcommand{\paraliewahory}{\mathbb{L}\mathrm{ie}(\theiwahori)}
\newcommand{\compactelements}{\mathfrak{C}}
\newcommand{\springertotalspace}{\widetilde{\compactelements}}
\newcommand{\thespringermap}{\mathfrak{p}}

\newcommand{\targetlattice}{\Xi}
\newcommand{\sourcelattice}{\Omega}

\DeclareMathOperator{\speciallineargroup}{SL}
\newcommand{\sltwo}{\speciallineargroup_2}
\newcommand{\bplustilde}{\widetilde{B^+}}
\newcommand{\bnplustilde}{\widetilde{B^+_n}}
\newcommand{\bplus}{{B^+}}
\newcommand{\bnplus}{{B^+_n}}
\newcommand{\canlift}[1]{\left[#1\right]}
\newcommand{\pararcspace}{\mathbb{L}^+}

\newcommand{\xto}[1]{\xrightarrow{#1}}

\newcommand{\into}{\hookrightarrow}
\newcommand{\onto}{\twoheadrightarrow}

\newcommand{\indconstructible}{\mathcal{D}}
\newcommand{\constructible}{\mathcal{D}_{\operatorname{c}}}
\newcommand{\prl}{\operatorname{Pr}^L_{\operatorname{st},\ell}}

\begin{document}
\begin{abstract}
This paper extends the affine Springer theory developed by Bouthier, Kazhdan, and the second author (see \texorpdfstring{\cite{bouthier2022perverse}}{[BKV]}) to the mixed characteristic case. In particular, we introduce a theory of perfectly placid perfect \texorpdfstring{$\infty$}{infinity}-stacks and establish their dimension theory. Furthermore, we prove that, in the Witt vector setting, the Chevalley morphism between arc spaces is flat.
\end{abstract}
\maketitle
\setcounter{tocdepth}{1}
\tableofcontents


\section{Introduction}


\begin{para}{\bf The affine Springer correspondence.}
In \cite{bouthier2022perverse}, the authors presented an affine version of the Grothendieck–Springer resolution, whose fibres are the affine Springer fibres.

Let $\thegroup$ be a connected reductive group over $\resfield$, where $\resfield$ is an algebraically closed field. Then $\arcspace(\thegroup)$ (resp. $\loopspace(\thegroup)$) is a group scheme (resp. group ind-scheme) defined over $\resfield$, representing the functor of points $R \mapsto \thegroup(R[[t]])$ (resp. $R \mapsto \thegroup(R((t)))$). We have corresponding definitions for the Lie algebra $\thelie$ of $\thegroup$: $\arcspace(\thelie)$ and $\loopspace(\thelie)$. The Iwahori subgroup (resp. subalgebra) is a certain closed subgroup scheme $\iwahory \subset \arcspace(\thegroup)$ (resp. $\liewahory$), with rational points corresponding to the classical Iwahori subgroup (see \cref{wittloopandarc:def:iwahori}).

We consider the affine flag variety $\flag := \loopspace(\thegroup) / \iwahory$ 
and the Chevalley morphism $\thechevalleymap : \thelie \to \thelie // \thegroup =: \chevalley$.
Denote the preimage $\compactelements := \loopspace(\thechevalleymap)^{-1}(\arcspace(\chevalley))\subset\loopspace(\thelie)$.
Let
\[
\springertotalspace := \explicitset{(g\iwahory, \gamma) \in \flag \times \compactelements \mid \theadjoint_{g^{-1}}(\gamma) \in \iwahory},
\]
then the affine Grothendieck–Springer resolution is the projection $\thespringermap:\springertotalspace \to \compactelements$. It is $\loopspace (G)$-equivariant, and we denote by $\overline{\thespringermap} : [\springertotalspace/\loopspace(\thegroup)] \to [\compactelements/\loopspace(\thegroup)]$ the induced map between quotient stacks. 

Assume that $\resfield$ is of characteristic $p> 2h$, where $h$ is the Coxeter number of $\thegroup$. Under this assumption,\footnote{The assumption on the characteristics is weaker in the original paper, but it was already corrected in \cite{bouthier2025perversity}.} it is shown in \cite{bouthier2022perverse}  that this Grothendieck–Springer resolution is semi-small, and it is deduced that the Grothendieck–Springer sheaf $\mathcal{S} := \overline{\thespringermap}_!(\omega_{[\springertotalspace/\loopspace(\thegroup)]})$ is perverse for some $t$-structure.
\end{para}

\begin{para}{\bf Goal.}
    The goal of this work to expand the results of \cite{bouthier2022perverse} to the mixed characteristic setting, as envisioned in \cite[\S 0.5.3]{bouthier2022perverse}.
    
    We use the Witt vector arc space (loop space) construction defined in \cite{zhu2017affine} and further developed in \cite{bhatt2017projectivity}. This construction is taking a scheme $X$ defined over $\wittzp(\resfield)$ 
    ({\em the Witt vectors} of $\resfield$) to the perfect scheme (perfect ind-scheme) over $\resfield$ with functor of points $R\mapsto X(\wittzp(R))$ ($R\mapsto X(\wittzp(R)[1/p])$).

    Since the current paper follows the structure of \cite{bouthier2022perverse} very closely, we try  to emphasis the differences between our work and \cite{bouthier2022perverse} and to not repeat the arguments that are the same.
\end{para}

Let $\theextendedweylgroup$ be the extended affine Weyl group of $G$. 
The main result of the paper is the following:

\begin{theorem}[\Cref{mainthm:mainthm}]
    \begin{mylist}
        \item The Witt affine Grothendieck--Springer sheaf \( \mathcal{S}_\bullet \) (defined in \Cref{mainthm:def:springer_sheaf} analogously to the equal characteristic case) is perverse with respect to the $t$-structure described in \cref{mainthm:def:tstructbullet}. 

        \item There is a natural isomorphism 
        \[
            \End(\mathcal{S}_\bullet) \simeq \qlbar[\theextendedweylgroup].
        \]

        \item Over the topologically nilpotent locus, the sheaf \( \mathcal{S}_{\bullet,\topnilp} \) is perverse with respect to the $t$-structure described in \cref{mainthm:def:tstructbullettopnilp}.
    \end{mylist}
\end{theorem}

The biggest difference between this work and \cite{bouthier2022perverse} is the {proof of the} following theorem (analogous to \cite[Theorem 3.4.7 and Corollary 3.4.8]{bouthier2022perverse}), shown in \cref{flatnessofchevalley} using a construction from \cite{bando2023derived}.

\begin{theorem}\label{thmintro:chevalleyflat}[\cref{codimofstrata:prop:finitechevalleyflat} and \cref{codimofstrata:prop:chevalleyflat}] Suppose $\operatorname{char}\resfield>2h$, where $h$ is the Coxeter number of $\thelie$.
Then the Chevalley morphisms between the arc spaces (see \cref{wittloopandarc:def:arcofnu} for definitions)
\[
        \thechevalleymap \colon \arcspace(\thelie) \to \arcspace(\chevalley)
        \quad \text{and} \quad
        \theiwahorychevalleymap \colon \liewahory \to \arcspace(\chevalley)
\]
    and their truncated versions:
\[
        \thechevalleymap_n \colon \arcspace_n(\thelie) \to \arcspace_n(\chevalley)
        \quad \text{and} \quad
        \theiwahorychevalleymap_n \colon \liewahory_n \to \arcspace_n(\chevalley)
\]
are flat.
\end{theorem}

The proof in \cite[Appendix B.4]{bouthier2022perverse} uses global methods that are not available in the mixed characteristic setting. Instead, following we use a {\em family of the arc spaces} over $\projectivespace^1$ (see \cref{def:jetandarc}), introduced in \cite{bando2023derived}. This family interpolates between the mixed and equal characteristic arc spaces: for every $x\in \mathbb{P}^1$ with $x\neq\infty$, its fiber is isomorphic to the mixed characteristic arc space, while the fiber over $x=\infty$ is the (perfection of) the equal characteristic arc space. Using this, we deduce \cref{thmintro:chevalleyflat} from its equal characteristic version by a generic flatness argument.

\begin{para}{\bf Working over perfect schemes.}
In this work, we define perfect $\infty$-stacks as sheaves on the \'etale site of perfect affine schemes over $\overline{\mathbb{F}}_p$ with values in the $\infty$-category of spaces, which we (following recent work of Dustin Clausen and Peter Scholze) call {\em anima} to avoid confusion with the notion of algebraic spaces. This setting forces us to change some of arguments of \cite{bouthier2022perverse}: 
\begin{mylist}
    \item \label{circomvanting_tangent:codimensions} The tangent space of a perfect scheme is always zero (see \cref{perfschemes:prop:notangent}), so the proof of \cite{goresky2006codimensions} for the codimensions of the strata in the GKM stratification fails in our setting. We circumvent this in \cref{appendixcodim}, presenting a different proof that uses an arc space version of the inverse function theorem.

    \item In \cite{bouthier2022perverse}, the construction of the stratification of $\arcspace(\chevalley)$ 
    is based on \cite[Proposition~3.11]{bouthier2022perverse}, which in turn is based on the work \cite{bouthier2022torsors}, which only applies to the equal characteristic case. Though in principle it should be possible to carry out the whole argument to the mixed characteristic case, we propose a different proof of the corresponding result \cref{stratoftorus:chevalleyisetale}, repeating the argument of \cite[Theorem 7.2.5]{bouthier2022perversesheavesinfinitedimensionalstacksv4}. 

    \item In \cite{bouthier2022perverse}, a notion of a strongly pro-smooth morphism is introduced. A morphism of schemes $f : X \to Y$ is strongly pro-smooth if one can exhibit $X$ as a filtered limit of smooth schemes finitely presented over $Y$ with smooth affine transition maps. An important ingredient (see \cite[\S 1.1.6]{bouthier2022perverse}) is that there exists a canonical presentation of strongly pro-smooth morphisms in this form. However, this argument does not generalize to strongly pro-perfectly smooth morphisms. As a result, certain arguments in the dimension theory yield weaker results (see \cref{dimthry:prop:pfpfbetweenplacidpresentedcorollaries}). Moreover, we provide a different proof (see \cref{perversetstructures:prop:welldefined}) for the uniqueness of the $t$-structure on perfectly placid stacks, since the proof of \cite[Proposition 6.3.1]{bouthier2022perverse} relies on the canonical smooth presentation.
\end{mylist}
\end{para}

\begin{para}{\bf Some other differences.}
There are other, smaller differences between this work and \cite{bouthier2022perverse} that we wish to highlight: 
\begin{mylist}
 \item Perfect schemes are always reduced. Consequently, there is no notion of a reduced perfect $\infty$-stack. This makes some of the definitions and results cleaner in this setting. For example:
    \begin{mylist}
        \item The notion of the complement of an open $\infty$-substack $\varstack{X} \smallsetminus \varstack{U}$ (see \cref{stratification:stratification}) is cleaner compared to \cite[\S2.4.1]{bouthier2022perverse}.
        \item In \cref{wittloopandarc:example:loopofGm}, we show that there is a natural isomorphism of {perfect} group schemes
        \[
        \loopspace (\multiplicativegroup) \simeq \ints \times \arcspace( \multiplicativegroup), 
        \]
        which holds only after reduction in the non-perfect setting.
    \end{mylist}

    \item In perfect schemes, the notion of universal homeomorphism coincides with that of isomorphism (see \cite[Lemma 3.8]{bhatt2017projectivity}). As a consequence, if a morphism $f : X \to Y$ between perfect schemes is universally open or universally closed, then it is an isomorphism if and only if it induces a bijection on $K$-points, $f : X(K) \to Y(K)$, for every algebraically closed field $K$ (see \cref{perfschemes:prop:graddream}). This slightly simplifies certain arguments, such as our proof of \cref{wittaffinrspringer:gneral:looppreservesquot} (compare \cite[\S B.2]{bouthier2022perverse}) and the proof of \cref{wittloopandarc:example:loopofGm}.
    
    \item Working with Witt vectors rather than power series necessitates re-verifying some arguments using Teichm\"uller representatives of Witt vectors. All of these proofs carry over (see, for example, \cref{stratoftorus:twistedstrataresmooth}).
\end{mylist}
\end{para}

\begin{para}{\bf Organization of the paper.}
\begin{mylist}
    \item We begin by setting up { a general theory of} perfect algebraic geometry and \emph{perfect $\infty$-stacks} in \cref{perfschemes} and \cref{perfstacks}. Next, in \cref{perfplacidstacks}, we develop a theory of perfectly placid perfect $\infty$-stacks and perfectly smooth morphisms between them. As in the corresponding definitions in the usual setting (see \cite[\S1]{bouthier2022perverse}), these can be viewed as perfect $\infty$-stacks and morphisms that ``up to a perfection'' are only singular in 
    finite codimension.

    \item Then, in \cref{dimthry}, we expand the dimension theory developed in \cite[\S2]{bouthier2022perverse} to the setting of perfect $\infty$-stacks and we introduce equidimensional morphisms between perfectly placid perfect $\infty$-stacks. This allows us to define placid stratifications and (semi-)small morphisms in \cref{stratification}.

    \item In \cref{GKMstrat}, \cref{Proofofchevalleyisetale}, and \cref{wittaffinrspringer}, we study the geometry of the morphism $\overline\thespringermap : [\springertotalspace/\loopspace(\thegroup)] \to [\compactelements/\loopspace(\thegroup)]$ and show that it is semi-small over the regular semisimple locus.

    \item In \cref{perversetstructures}, we develop the analogue of the sheaf theory and perverse $t$-structures constructed in \cite[Part~3]{bouthier2022perverse} for perfect $\infty$-stacks.

    \item In \cref{mainthm}, we formulate and prove the main theorem of this work (\cref{mainthm:mainthm}).

    \item In \cref{flatnessofchevalley}, we prove \cref{thmintro:chevalleyflat}.

    \item In \cref{appendixcodim}, we calculate the codimensions of the strata of the GKM stratification as described in \cref{circomvanting_tangent:codimensions}.
\end{mylist}
\end{para}

\begin{para}{\bf Acknowledgements.}
We thank Alexis Bouthier and David Kazhdan for their 
collaboration on the paper \cite{bouthier2022perverse} on which this work is based. We also thank Peter Scholze for his suggestion to extend the results of \cite{bouthier2022perverse} to the setting of perfect $\infty$-stacks. This research was partially supported by the ISF grants 2091/21 and 2889/25 of the second author. Most of the work 
has been carried out while the first author was a graduate student at the Hebrew University of Jerusalem. 
\end{para}

\section{Perfect schemes}\label{perfschemes}

\begin{para}{\bf Notations.}
    \begin{mylist}
        \item Let $\resfield$ be an algebraically closed field of characteristic~$p$.
        \item Let $\aff$, $\sch$, and $\algsp$ denote the ($1$-)categories of affine schemes, schemes, and algebraic spaces over~$\resfield$, respectively, and let $\sch^{\operatorname{qcqs}}\subset\sch$ be the subcategory of quasi-compact and quasi-separated schemes. 
    \end{mylist}
\end{para}

\begin{para}{\bf Perfect schemes.} \label{perfschemes:perfschemes}
    \begin{mylist}
        \item A scheme $U$ is said to be \emph{perfect}, if its absolute Frobenius morphism $\Phi=\Phi_p \colon U \to U$ is an isomorphism.
        \item Let $\persch$ denote the full subcategory of $\sch$ consisting of perfect schemes. The inclusion functor $\persch \into \sch$ admits a right adjoint, sending a scheme $X$ to its perfection
        \[
        X_{\perf} := \lim \left( \cdots \overset{\Phi}{\longrightarrow} X \overset{\Phi}{\longrightarrow} X \right),
        \]
        where the limit is taken in the category of schemes (rather than in $\sch$ - the categories of schemes over $\resfield$, where the diagram is not defined).
        For an overview of perfect schemes, see \cite[\S 3]{bhatt2017projectivity}.
    \end{mylist}
\end{para}

\begin{para}{\bf Universal homeomorphisms.}
    \begin{mylist}
        \item The canonical morphism $X_{\perf} \to X$ is a universal homeomorphism.
        \item A scheme $X$ is perfect if and only if every universal homeomorphism $X' \to X$ with $X'$ reduced is an isomorphism (see \cite[Lemma 14.2.1]{barwick2020exodromy}).
    \end{mylist}
\end{para}

\begin{lemma}\label{perfschemes:prop:graddream}
    Let $f \colon X \to Y$ be a morphism of perfect schemes. Then the following conditions are equivalent:
    \begin{mylist}
        \item The morphism $f$ is an isomorphism.
        \item The morphism $f$ is a universal homeomorphism.
    \end{mylist}
    Moreover, if $f$ is either universally closed or universally open, then these conditions are also equivalent to:
    \begin{mylist}[(c)]
        \item For any algebraically closed field $K$, the induced map on $K$-points $f \colon X(K) \to Y(K)$ is a bijection.
    \end{mylist}
\end{lemma}

\begin{proof}
    The equivalence (a) $\iff$ (b) is \cite[Lemma 3.8]{bhatt2017projectivity}.  
    The implication (a) $\Rightarrow$ (c) is immediate.

    We prove (c) $\Rightarrow$ (b). Suppose $f$ is universally open or universally closed. To show that $f$ is a universal homeomorphism, it suffices to prove that $f$ is universally bijective on the underlying topological space.

    Since $f$ induces a bijection on $K$-points for all algebraically closed fields $K$, it is universally bijective: points of a scheme correspond to equivalence classes of maps from algebraically closed fields (see \cite[\href{https://stacks.math.columbia.edu/tag/01J9}{Tag 01J9}]{stacks-project}).
\end{proof}

\begin{proposition}\label{perfschemes:prop:topologicalinvarianceofetale}
    {\bf Étale morphisms and universal homeomorphisms} (\cite[Prop.~14.1.7]{barwick2020exodromy}).  
    Let $f \colon X \to Y$ be a universal homeomorphism of schemes, and let $Y' \to Y$ be an étale morphism. Then the base change $X \times_Y Y' \to X$ is étale. Moreover, the induced morphism of étale sites
    \[
    f^* \colon Y_{\acute{e}t} \to X_{\acute{e}t}
    \]
    is an equivalence of sites. In particular, the canonical morphism $X_{\perf} \to X$ induces an equivalence of étale topoi:
    \[
    \Shv(X_{\acute{e}t}) \simeq \Shv((X_{\perf})_{\acute{e}t}).
    \]
\end{proposition}

\begin{proposition}{\bf Properties preserved by the perfection functor.} \label{perfschemes:prop:propertiesofperfschemespreserved}
    Let $f \colon X \to Y$ be a morphism of algebraic spaces. If $f$ has one of the following properties, then its perfection $f_{\perf} \colon X_{\perf} \to Y_{\perf}$ has the same property:
    \begin{mylist}
        \item Quasi-compact or quasi-separated;
        \item Affine;
        \item Étale;
        \item Flat;
        \item (Universally) open or closed;
        \item closed, open, or locally closed embedding;
        \item Surjective (in the étale topology).
    \end{mylist}
\end{proposition}

\begin{proof}
    The cases of quasi-compact, quasi-separated, affine, étale, flat, universally closed, and being a closed, open, or locally closed embedding are treated in \cite[Lemma~3.4]{bhatt2017projectivity}.

    The statement for universally open follows from the fact that the perfection morphism is a universal homeomorphism and that the perfection functor preserves pullbacks.

    The case of surjectivity in the étale topology follows from \cite[\href{https://stacks.math.columbia.edu/tag/05VM}{Tag 05VM}]{stacks-project}.
\end{proof}

\begin{definition}
    Let $f \colon X \to Y$ be a morphism of perfect schemes. We say that $f$ has one of the properties listed in \cref{perfschemes:prop:propertiesofperfschemespreserved}, if it has that property as a morphism of schemes.

    In other words, we do not prefix such adjectives with “perfectly”; for example, we say that $f$ is flat, not “perfectly flat”, whenever this is true as a morphism of schemes.
\end{definition}

\begin{para}\label{perfschemes:prop:notangent} {\bf Tangent spaces of perfect schemes.}
    Let $U = \spec R$ be a perfect affine scheme. Since $R$ is perfect, for any maximal ideal $\mathfrak{m}\subset R$ we have that $\mathfrak{m} = \mathfrak{m}^p$ and so $\mathfrak{m}= \mathfrak{m}^2$. Thus, for $X\in \persch$, the tangent space $T_x X = 0$ for any $x \in X$. 
\end{para}

\begin{para}\label{perfschemes:prop:perfarered} {\bf Perfect schemes are reduced.}
    Let $U = \spec R$ be a perfect affine scheme, then $R$ is reduced, as the $p^n$-power map is an isomorphism for any $n$. Thus, any perfect scheme is reduced.
\end{para}

\begin{para}\label{perfschemes:def:perfoffinitetype}{\bf Perfectly finitely presented.}
    \begin{mylist}
        \item Note that 
        \[
        \aff \simeq \pro(\aff^{\fp}), \quad \sch^{\operatorname{qcqs}} \simeq \pro(\sch^{\fp}), 
        \]
        where $\aff^{\fp}, \sch^{\fp}$ are the categories of finitely presernted affine schemes  and of finitely presented schemes over $\spec \resfield$ (equivalently, schemes of finite type).
        Equivalently, the co-compact objects are the full subcategories $\aff^{\fp} \subset \aff$ and $\sch^{\fp} \subset \sch^{\qcqs}$ of affine schemes and schemes finitely presented over $\spec \resfield$, and they are co-generating. 

        \item  A morphism of perfect schemes $f \colon X \to Y$ is said to be (locally) \emph{perfectly finitely presented}, if (locally) it is isomorphic to the perfection of a finitely presented morphism $f_0 \colon X_0 \to Y$, i.e., if $f \simeq (f_0)_{\perf}$. We denote by $\peraff^{\pfp}$ and $\persch^{\pfp}$ the full subcategories of $\peraff$ and $\persch$, consisting of objects perfectly finitely presented over $\spec \resfield$.
        
        \item Perfectly finitely presented perfect schemes are co-compact objects in $\peraff$ and $\persch^{\qcqs}$. In fact, every co-compact object in these categories is perfectly finitely presented, and these categories are co-generated by cocompact objects. Hence, we obtain equivalences of categories:
        \[
        \peraff \simeq \pro(\peraff^{\pfp}), \quad \persch^{\operatorname{qcqs}} \simeq \pro(\persch^{\pfp}).
        \]
    \end{mylist}
\end{para}
\begin{proposition}\label{perfschemes:prop:existsfinitetypemodel}%
Let $f \colon X \to Y$ be a morphism in $\persch^{\pfp}$. Then there exists a morphism $f_0 \colon X_0 \to Y_0$ of finitely presented schemes over $\resfield$ such that $f \simeq (f_0)_{\perf}$.
\end{proposition}
\begin{proof}
    This follows from the proof of \cite[Proposition 3.13]{bhatt2017projectivity}. Namely, using \cite[Proposition 3.13]{bhatt2017projectivity} we choose a finitely presented scheme $Y_0$ such that $Y\simeq (Y_0)_{\perf}$. We can write $X = \lim_i X_i$ as a cofiltered limit of finitely presented $\resfield$-schemes with transition maps being universal homeomorphisms. Since $Y_0$ is cocompact, for some $i$ we get a factorization $X_i\to  Y_0$. We take $X_0:= X_i$ and get the desired result. 
\end{proof}

\begin{proposition}\label{prop:flatlocusopen}
    A morphism $f:X\to Y$ between quasi-compact perfectly finitely presented schemes has an open flat locus.
\end{proposition}
\begin{proof}
    Apply \cref{perfschemes:prop:existsfinitetypemodel} to get a model $f_0 :X_0 \to Y_0$ which is finitely presented over $\resfield$. There is an open subset $U_0 \subset X$ such that $f_0|_{U_0}$ is flat. By \cref{perfschemes:prop:propertiesofperfschemespreserved} for open immersions and flat morphisms, $U:= (U_0)_{\perf}$ is an open subscheme of $X$ where $f$ is flat.


\end{proof}

\begin{theorem}[Critère de platitude par fibres]\label{thm:critflatfibre}
    Let $f \colon X \to Y$ be a morphism of perfect schemes over a perfect base scheme $S$. Let $x \in X$ and let $s \in S$, $y \in Y$ be its images. Suppose that $f$ is the perfection of a morphism   $f_0:X_0\to Y_0$ of schemes  over $S_0$, which are Noetherian over $\resfield$. Assume that $X_0$ is flat over $S_0$ at $x$, and that the base change
    \[
    f_0 \times_{S_0} \spec k(s) \colon X_0 \times_{S_0} \spec k(s) \to Y_0 \times_{S_0} \spec k(s)
    \]
    is flat at $x$.

    Then $Y$ is flat over $S$ at $y$, and $f \colon X \to Y$ is flat at $x$.
\end{theorem}

\begin{proof}
    By \cite[\href{https://stacks.math.columbia.edu/tag/039C}{Tag 039C}]{stacks-project}, $Y_0$ is flat over $S_0$ at $y$, and $f_0 \colon X_0 \to Y_0$ is flat at $x$. The claim then follows from \cref{perfschemes:prop:propertiesofperfschemespreserved}.
\end{proof}

\begin{definition}\label{perfschemes:def:perfectlysmooth}
    (\cite[Def.~A.25]{zhu2017affine})  
    A morphism $f \colon X \to Y$ between perfect schemes perfectly finitely presented over $\resfield$ is called \emph{perfectly smooth}, if for every $x \in X$ there exists an étale neighborhood $U \to X$ of $x$ fitting into a commutative diagram
    \[
    \begin{tikzcd}
        U \arrow[r] \arrow[d] & Y \times \affinespace^n_{\perf} \arrow[d] \\
        X \arrow[r, "f"] & Y, 
    \end{tikzcd}
    \]
    where the upper horizontal arrow is étale, and the right vertical arrow is the projection.
\end{definition}

The following was mentioned in \cite[\S 2.1.2]{van2024mod} without proof; these properties are straightforward to verify.

\begin{proposition}\label{perfschemes:prop:propertiesperfectlysmooth}%
    {\bf Properties of perfectly smooth morphisms.}
    The class of perfectly smooth morphisms satisfies:
    \begin{mylist}
        \item Étale morphisms and isomorphisms are perfectly smooth.
        \item The composition of perfectly smooth morphisms is perfectly smooth.
        \item The class of perfectly smooth morphisms is stable under pullbacks.
    \end{mylist}
\end{proposition}

\begin{remark}
     Cerfectly smooth morphisms can not be smooth, unless the diagram in \cref{perfschemes:def:perfectlysmooth} can be chosen with $n=0$, in which case the morphism is étale.
\end{remark}


\begin{definition}\label{perfschemes:def:perfectlyproper}
    A morphism $f \colon X \to Y$ of perfect schemes is called \emph{perfectly finitely presented proper} (or \emph{pfp-proper}), if it is perfectly finitely presented, separated and universally closed.
\end{definition}

\begin{remark}
    Let $f \colon X \to Y$ be a perfectly finitely presented morphism of perfect stacks. Let $f_0\colon X_0\to Y_0$ be as in \Cref{perfschemes:prop:existsfinitetypemodel}. Then $f$ is pfp-proper if and only if $f_0$ is proper (see \cite[Corollary 3.15]{bhatt2017projectivity}). 
\end{remark}

\section{Perfect \texorpdfstring{\( \infty \)}{infinity}-stacks}\label{perfstacks}

\begin{para}{\bf Preliminaries on $\infty$-stacks.}\label{notationsandconventions:algebraicgeometry}
    Let $\catofanima$ denote the $\infty$-category of {\em anima} commonly referred to as {\em $\infty$-groupoids} or {\em spaces}.
    Let $\prestacks := \Fun(\aff^{\op}, \catofanima)$ denote the $\infty$-category of $\infty$-prestacks.
    \begin{mylist}
        \item An object $\varstack{X} \in \prestacks$ is called an \emph{$\infty$-stack}, if it is a sheaf for the étale topology (see \cite[Def.~6.2.2.7]{lurie2009higher}). We denote the full $\infty$-subcategory of $\infty$-stacks by $\stacks$.
        \item\label{notationsandconventions:def:space} An $\infty$-stack is called a \emph{space}, if it takes values in $\catofsets \subset \catofanima$.
        \item There is a fully faithful embedding $\sch \into \stacks$ given by the restricted Yoneda functor, sending a scheme to its functor of points. We identify the category of schemes with its essential image under this embedding.
    \end{mylist}
\end{para}

\begin{definition}\label{perfstacks:def:coverings}{\bf Coverings.}  
    A morphism $\varstack{X} \to \varstack{Y}$ of $\infty$-stacks is called a \emph{covering}, if it is an effective epimorphism in the sense of \cite{lurie2009higher}. More precisely, by \cite[\S 7.2.1.14]{lurie2009higher}, effective epimorphisms in an $\infty$-category of sheaves are characterized as follows: for any morphism $U \to \varstack{Y}$ with $U$ affine, there exists an étale cover (that is, an étale surjective map) $U'\to U$ such that  the composition $U'\to U \to \varstack{Y}$ factors through $\varstack{X}$.
\end{definition}

\begin{definition}{\bf Perfect $\infty$-stacks.}\label{perfstacks:def:perfstacks}
    We equip $\peraff$ with the étale topology.  
    A \emph{perfect $\infty$-stack} is a sheaf on the étale site of perfect affine schemes. We denote by $\perstacks := \Shv(\peraff)$ the $\infty$-category of perfect $\infty$-stacks.
\end{definition}

\begin{para}\label{perfstacks:def:perfstacksembedding}{\bf Perfect $\infty$-stacks as $\infty$-stacks.}  
The inclusion $i \colon \peraff \into \aff$ is a morphism of sites.
The induced restriction functor 
    \[
    (-)_{\perf} \colon \stacks \to \perstacks:=\Shv(\peraff)
    \]
    (which we call the \emph{perfectization}) admits a fully faithful left adjoint $\iota$.
    
    

\end{para}

\begin{remark}{\bf Limits of perfect $\infty$-stacks.}
    \begin{mylist}
        \item The fully faithful embedding $\perstacks \into \stacks$ preserves all colimits but does not preserve limits. We always work in the $\infty$-category $\perstacks$; thus, limits are taken in this $\infty$-category.
        
        \item However, the difference is exactly the application of the perfection functor. Since perfection is a right adjoint, it preserves all limits. Thus, for a diagram $F \colon I \to \perstacks$, we have
        \[
        \lim F \simeq (\lim (\iota\circ F))_{\perf}.
        \]
    \end{mylist}
\end{remark}

\begin{remark}\label{perfstacks:remark:noneedtowritered}{\bf Reduction issues do not arise for perfect $\infty$-stacks.}
\begin{mylist}
    \item Since there is no notion of a non-reduced perfect scheme, there is correspondingly no notion of a non-reduced perfect $\infty$-stack.
    \item Moreover, in \cite{bouthier2022perverse}, a class of \emph{reduced} $\infty$-stacks is considered. We observe that, by \cref{perfschemes:prop:perfarered}, the functor $\iota$ takes values in the full $\infty$-subcategory of reduced $\infty$-stacks.
\end{mylist}
\end{remark}

\begin{para}{\bf Perfect algebraic spaces.} 
    Note that if an $\infty$-stack $X$ is representable by an algebraic space, then $X_{\perf}$ is also representable by an algebraic space (by \cref{perfschemes:prop:topologicalinvarianceofetale}). Thus, the definition of algebraic spaces applies consistently, whether considered as sheaves on affine schemes or on perfect affine schemes.
\end{para}

\begin{para}{\bf Properties of perfect $\infty$-stacks.}
    \begin{mylist}
        \item\label{perfstacks:def:pfprepresentable}  A morphism $f:\varstack{X}\to \varstack{Y}$ of perfect $\infty$-stacks is {\em (locally) pfp-representable}, if for any morphism $U\to \varstack{Y}$ with $U$ a perfect affine scheme, $f\times_{\varstack{Y}}U$ is a (locally) perfectly finitely presented morphism of perfect algebraic spaces.
        \item\label{perfstacks:def:uorepresentable}  A morphism $f:\varstack{X}\to \varstack{Y}$ of perfect $\infty$-stacks is {\em universally open representable}, if for any morphism $U\to \varstack{Y}$ with $U$ a perfect affine scheme, $f\times_{\varstack{Y}}U$ is a universally open morphism of perfect algebraic spaces.
        \item\label{perfstacks:def:pfpembeddings} A morphism $f:\varstack{X}\to \varstack{Y}$ of perfect $\infty$-stacks is a {\em pfp-open/closed/locally closed embedding}, if for any morphism $U\to \varstack{Y}$ with $U$ a perfect affine scheme, $f\times_{\varstack{Y}}U$ is a perfectly finitely presented open/closed/locally closed embedding of perfect algebraic spaces.
    \end{mylist}
\end{para}

\begin{para}{\bf Properties that are étale-local on the base.} 
    We say a class (P) of morphisms from perfect $\infty$-stacks to perfect affine schemes is \emph{étale-local on the base}, if for every morphism $f \colon \varstack{X} \to Y$, where $\varstack{X}$ is a perfect $\infty$-stack and $Y$ is a perfect affine scheme, and every étale covering $Y' \to Y$, the morphism $f$ is in (P) if and only if the base change $f\times_Y Y'$ is in (P).
\end{para}

\begin{example}\label{perfstacks:prop:examplesofpropertiesetalelocal}
    The class of isomorphisms is étale-local on the base. Similarly, the class of universally open morphisms is étale-local on the base. (Indeed, this follows from the non-perfect case).
\end{example}

\begin{para}\label{perfstacks:def:locallyschematic}{\bf Locally schematic morphisms.}
    \begin{mylist}
        \item The class of schematic morphisms is not étale-local on the base.
        \item We say that a morphism of perfect $\infty$-stacks $f \colon \varstack{X} \to \varstack{Y}$ is \emph{locally schematic},  if for any morphism $Y \to \varstack{Y}$ with $Y$ an affine scheme, there exists an étale cover $Y'\to Y$ such that the  fiber product $\varstack{X} \times_{\varstack{Y}} Y'$ is a perfect scheme.
        \item The class of locally schematic morphisms from perfect $\infty$-stacks to schemes $\varstack{X}\to Y$ is étale-local on the base.
    \end{mylist}
\end{para}

\begin{para}\label{perfstacks:def:indpfpproper}{\bf Ind pfp-proper algebraic spaces.}
    \begin{mylist}
        \item A morphism of $\infty$-stacks $\varstack{X} \to \varstack{Y}$ is called \emph{ind pfp-proper}, if for every morphism $Y \to \varstack{Y}$ with $Y$ an affine scheme, the base change $\varstack{X} \times_{\varstack{Y}} Y \to Y$ can be written as a filtered colimit $\colim X_i$ over $Y$ of pfp-proper algebraic spaces, with transition maps given by closed embeddings.
        
        \item\label{perfstacks:def:indpfpproperlocschematic} A morphism $\varstack{X} \to \varstack{Y}$ is called \emph{ind pfp-proper-locally schematic},  if for every morphism $Y \to \varstack{Y}$ with $Y$ affine, the base change $\varstack{X} \times_{\varstack{Y}} Y \to Y$ can be written as a filtered colimit $\colim X_i$ over $Y$, where each $X_i \to Y$ is a pfp-proper and locally schematic morphism of algebraic spaces, and the transition maps are closed embeddings.

        \item The class of ind pfp-proper morphisms is étale-local on the base, by the same argument as in \cite[Remark 1.2.9(a)]{bouthier2022perverse}, where the analogous statement is proved for ind-fp-proper morphisms.
    \end{mylist}
\end{para}

\begin{proposition}\label{perfstacks:prop:graddream}
    Let $f:\varstack{X} \to \varstack{Y}$ be a morphism of perfect $\infty$-stacks satisfying one of the following:
    \begin{mylist}
        \item $f$ is locally schematic and universally open representable (see \cref{perfstacks:def:uorepresentable}).
        \item $f$ is ind pfp-proper-locally schematic (see \cref{perfstacks:def:indpfpproperlocschematic}).
    \end{mylist}
    Then $f$ is an equivalence if and only if for every algebraically closed field $K$, the induced map of anima
    \[
    f:\varstack{X}(K) \to \varstack{Y}(K)
    \]
    is an equivalence.
\end{proposition}

\begin{proof}
    Since isomorphisms are étale-local on the base (see \cref{perfstacks:prop:examplesofpropertiesetalelocal}), and every perfect $\infty$-stack admits a cover by perfect affine schemes, it suffices to assume that $\varstack{Y} = Y$ is a perfect affine scheme.

    In the first case, if $f$ is locally schematic and universally open, then (after an étale cover of $Y$) the source $\varstack{X}$ is a scheme. The claim then follows from \cref{perfschemes:prop:graddream} and étale descent for isomorphisms.

    In the second case, suppose $f$ is ind pfp-proper-locally schematic. Then $\varstack{X} \simeq \colim X_i$ is a filtered colimit over $Y$ of pfp-proper and locally schematic algebraic spaces with closed embeddings for transition maps. By quasi-compactness of the constructible topology on $Y$, there exists $i \gg 0$ such that $X_i \to Y$ is surjective, and hence $X_i \to \varstack{X}$ is a surjective closed embedding. It follows that $X_i \simeq \varstack{X}$, and so $f$ is a pfp-proper and locally schematic morphism. In particular, it is universally closed. The result then follows from \cref{perfschemes:prop:graddream}.
\end{proof}

\begin{para}{\bf Torsors.}
    \begin{mylist}
        \item Let $\varstack{X}$ be a perfect $\infty$-stack, and let $H$ be a perfect group space acting on $\varstack{X}$. We may form the quotient perfect $\infty$-stack $[\varstack{X}/H]$ via the bar construction.

        \item \label{perfstacks:prop:torsorequivalentcond} Let $f \colon \varstack{X} \to \varstack{Y}$ be a morphism of perfect $\infty$-stacks, and let $H$ be a perfect group space acting on $\varstack{X}$ over $\varstack{Y}$. Then the following conditions are equivalent (see \cite[\S 1.2.6 (c)-(d)]{bouthier2022perverse}, the argument there works in a general $\infty$-topos):
        \begin{mylist}
            \item The natural map $[\varstack{X}/H] \to \varstack{Y}$ is an equivalence.
            \item The morphism $f$ is an $H$-torsor: that is, $f$ is a covering (see \cref{perfstacks:def:coverings}) and the action morphism
            \[
            H \times \varstack{X} \to \varstack{X} \times_{\varstack{Y}} \varstack{X}, \quad (h, x) \mapsto (hx, x)
            \]
            is an equivalence.
        \end{mylist}

        \item For $U \in \peraff$, the groupoid $[\varstack{X}/H](U)$ classifies pairs $(\widetilde{U}, \phi)$ where $\widetilde{U} \to U$ is an $H$-torsor, and $\phi \colon \widetilde{U} \to \varstack{X}$ is an $H$-equivariant morphism.\footnote{We caution the reader unfamiliar with higher category theory that the anima of equivariant morphisms between spaces is generally not a disjoint union of contractible anima. Consequently, the quotient of a space by a group space will typically not be a space, but rather a $1$-stack (i.e., a stack in groupoids).}
    \end{mylist}
\end{para}

\begin{para}\label{perfstacks:prop:descentandtorsors}{\bf Descent for classes of morphisms which are \'etale local on the base.}
    Let (P) be a class of morphisms of perfect $\infty$-stacks, and suppose it is \'etale local on the base.

    Let $\varstack{X}, \varstack{Y}$ be perfect $\infty$-stacks equipped with an action of a perfect group space $H$. Suppose $f \colon \varstack{X} \to \varstack{Y}$ is an $H$-equivariant morphism in the class (P). Then the induced morphism
    \[
    [f/H] \colon [\varstack{X}/H] \to [\varstack{Y}/H]
    \]
    also lies in the class (P).
\end{para}

\begin{proof}
    See \cite[Lemma 1.2.7]{bouthier2022perverse}.
\end{proof}

\begin{proposition}\label{perfstacks:prop:torsoronpoints}{\bf Torsors detected on geometric points.}
    Let $f \colon \varstack{X} \to \varstack{Y}$ be a morphism of perfect $\infty$-stacks which is either:
    \begin{enumerate}
        \item locally schematic and universally open representable (see \cref{perfstacks:def:uorepresentable}), or
        \item ind-pfp-proper-locally schematic (see \cref{perfstacks:def:indpfpproperlocschematic}).
    \end{enumerate}
    Let $\Lambda$ be a discrete group acting on $\varstack{X}$ over $\varstack{Y}$. Suppose that for every algebraically closed field $K$, the induced morphism $\varstack{X}(K) \to \varstack{Y}(K)$ is a $\Lambda$-torsor in $\catofanima$. Then $f$ is a $\Lambda$-torsor.
\end{proposition}

\begin{proof}
    Let (P) be either the class of locally schematic and universally open morphisms or the class of ind-pfp-proper-locally schematic morphisms. The class (P) is stable under base change and composition.

    Moreover, (P) is \'etale local on the target: for the universally open case, this follows from \cref{perfstacks:prop:examplesofpropertiesetalelocal}; for the ind-pfp-proper case, see \cref{perfstacks:def:indpfpproper}(c).

    Consider the morphism $[f/\Lambda] \colon [\varstack{X}/\Lambda] \to [\varstack{Y}/\Lambda] \simeq \varstack{Y} \times [*/\Lambda]$. Since (P) is \'etale local on the target, this morphism lies in (P). Note that the morphism $[f/\Lambda]$ factors as a composition $[\varstack{X}/\Lambda] \to \varstack{Y}\to [\varstack{Y}/\Lambda]$. By \cite[11.1.19]{ravi2022rising}\footnote{The theorem is stated there for schemes, but its proof applies in any $\infty$-category with pullbacks.}, to show that the morphism $[\varstack{X}/\Lambda] \to \varstack{Y}$ is in (P), it is enough to show that $[f/\Lambda]$ and the diagonal of $\varstack{Y}\to [\varstack{Y}/\Lambda]$ are in (P).

    The diagonal morphism $\varstack{Y} \to \varstack{Y} \times_{[\varstack{Y}/\Lambda]}\varstack{Y}$,
    agrees with the unit section
    $\varstack{Y}\to \Lambda\times \varstack{Y}$
    given by $y \mapsto (1,y)$ under the identification $\varstack{Y} \times_{[\varstack{Y}/\Lambda]}\varstack{Y}\simeq\Lambda \times \varstack{Y}$. Thus the diagonal is also in (P). 
    

    Since $[\varstack{X}/\Lambda] \to \varstack{Y}$ is in (P) and is an equivalence on geometric points, we conclude by \cref{perfstacks:prop:graddream} that $f$ is a $\Lambda$-torsor.
\end{proof}

\section{Perfectly placid \texorpdfstring{\( \infty \)}{infinity}-stacks and perfectly smooth morphisms between them}\label{perfplacidstacks}

In this subsection, we use \cite[Appendix A.1-A.3]{bouthier2022perverse} to define the notions of perfectly placid $\infty$-stacks and perfectly smooth morphisms. We include 
self-contained versions of all definitions for the convenience of the reader.

\begin{definition}\label{perfplacid:def:perfplacidpresandstronglyprosmooth}{\bf Strongly pro-perfectly smooth morphisms and perfectly placid presentations.}
    \begin{mylist}
        \item A morphism \( f \colon X \to Y \) of perfect algebraic spaces is called \emph{strongly pro-perfectly smooth}, if \( X \) can be written as a filtered limit \( X \simeq \lim X_\alpha \), where each \( X_\alpha \to Y \) is perfectly finitely presented and perfectly smooth, and the transition maps \( X_\alpha \to X_\beta \) are perfectly finitely presented, perfectly smooth, and affine.

        \item Let $X$ be a perfect algebraic space. A \emph{perfectly placid presentation} of $X$ is an identification $X \simeq \lim X_\alpha$, where each $X_\alpha$ is a perfect algebraic space perfectly finitely presented over $\spec \resfield$, and the transition maps $X_\alpha \to X_\beta$ are perfectly smooth and affine.

        \item We say that a perfect algebraic space admits a {\em perfectly placid presentation}, if such a presentation exists.

        \item We say that a morphism of schemes is {\em strongly pro-perfectly smooth}, or that a scheme admits a {\em perfectly placid presentation}, if the corresponding properties hold when viewed as algebraic spaces.
    \end{mylist}
\end{definition}

\begin{remark}\label{perfplacid:prop:perfplacidpresandstronglyprosmoothinpro}
    Recall from \cref{perfschemes:def:perfoffinitetype} that
    \[
    \persch^{\operatorname{qcqs}} \simeq \pro( \persch^{\pfp}), \quad \text{and} \quad \peraff \simeq \pro (\peraff^{\pfp}).
    \]
    Let $\varcat{B}$ denote either $\persch^{\pfp}$ or $\peraff^{\pfp}$, and let $\mor_0(\varcat{B})$ be the class of perfectly smooth morphisms in $\varcat{B}$. Then the construction of classes of morphisms and objects in a pro-category, as formulated in \cite[A.3.1]{bouthier2022perverse}, applies to the pair $(\varcat{B}, \mor_0(\varcat{B}))$.

    The resulting class $\mor_0(\pro\varcat{B})$ is the class of strongly pro-perfectly smooth morphisms (see \cref{perfplacid:def:perfplacidpresandstronglyprosmooth}), and $\ob_0(\pro\varcat{B})$ is the class of schemes or affine schemes admitting a perfectly placid presentation.
\end{remark}

\begin{para}\label{perfplacid:def:perfplacidandperfsmooth}{\bf Perfectly smooth morphisms and perfectly placid $\infty$-stacks.}
    We apply the construction of \cite[\S A.2.4]{bouthier2022perverse} to the $\infty$-category $\perstacks$, with the classes of objects and morphisms described in \cref{perfplacid:def:perfplacidpresandstronglyprosmooth}. The conditions for the construction are satisfied by \cref{perfplacid:prop:perfplacidpresandstronglyprosmoothinpro} and \cite[A.3.2]{bouthier2022perverse}. We obtain classes of morphisms \( \mor_\infty \) and objects \( \ob_\infty \), defined as the smallest classes satisfying the following properties:
    \begin{mylist}
        \item Let \( \{X_\alpha\} \) be a collection of schemes admitting perfectly placid presentations. Then \( \coprod X_\alpha \in \ob_\infty \).
        \item Let \( \{X_\alpha \to Y\} \subset \mor_0 \) be a collection of morphisms in \( \persch \). Then \( \coprod X_\alpha \to Y \in \mor_\infty \).
        \item The class \( \mor_\infty \) is stable under base change, composition, and contains all isomorphisms.
        \item Let \( \varstack{X} \in \ob_\infty \), and let \( \varstack{X} \to \varstack{Y} \in \mor_\infty \) be a covering (see \cref{perfstacks:def:coverings}). Then \( \varstack{Y} \in \ob_\infty \).
        \item Let \( \varstack{X} \to \varstack{Y} \) be a covering in \( \mor_\infty \), and suppose the composition \( \varstack{X} \to \varstack{Y} \to \varstack{Z} \) lies in \( \mor_\infty \). Then \( \varstack{Y} \to \varstack{Z} \in \mor_\infty \).
    \end{mylist}
    A morphism \( f: \varstack{X} \to \varstack{Y} \) in \( \mor_\infty \) is called a \emph{perfectly smooth morphism}, and a perfect \( \infty \)-stack \( \varstack{X} \in \ob_\infty \) is called \emph{perfectly placid}.
\end{para}

\begin{para}\label{perfplacid:prop:canonicalsmoothcover}{\bf The canonical perfectly smooth cover} (see \cite[Theorem A.2.8]{bouthier2022perverse}).
    Let $\varstack{X}$ be a perfectly placid $\infty$-stack. We define a full $\infty$-subcategory $\mathcal{J}_{\varstack{X}} \subset (\persch)_{/\varstack{X}}$ of the slice $\infty$-category. The objects of $\mathcal{J}_{\varstack{X}} $ are perfect schemes $Y = \coprod Y_i$  such that each $Y_i$ is a perfect affine scheme which admits a perfectly placid presentation, equipped with a perfectly smooth morphism $Y \to \varstack{X}$. The morphisms of $\mathcal{J}_{\varstack{X}} $ are morphisms over $\varstack{X}$ which are  strongly pro-perfectly smooth on connected components. Then $\mathcal{J}_{\varstack{X}}$ is a filtered $\infty$-category, and there is a canonical equivalence
    \[
    \varstack{X} \simeq \colim_{Y \in \mathcal{J}_{\varstack{X}}} Y.
    \]
\end{para}

\begin{remark}
    In \cite{bouthier2022perverse} (see \cite[\S 1.3]{bouthier2022perverse} for a particular case and \cite[\S A.1]{bouthier2022perverse} for the general categorical framework), the authors construct the classes $\ob_\infty$ and $\mor_\infty$ by an explicit inductive procedure. We omit this construction here, but note  that it works exactly the same way here.
\end{remark}

\section{Dimension theory}\label{dimthry}

In this section, we will define notions of weakly equidimensional, equidimensional, and uo-equidimensional between perfectly placid $\infty$-stacks. We apply the categorical framework developed in \cite[Appendix A.4]{bouthier2022perverse}. The reader who wishes to understand the proofs of this section is advised to read the self-contained account in \cite[Appendix A]{bouthier2022perverse}. We include a self-contained exposition of all definitions for convenience of the reader.

\begin{para}{\bf Dimension function.} Let $f:X\to Y$ be a morphism of perfect algebraic spaces, perfectly finitely presented over \( \spec\resfield \). 
        We associate to $f$ a function $\dimfunc_f:|X|\to \ints$, given by 
        \[
        \dimfunc_f(x):= \dim_x(X)-\dim_{f(x)}(Y).
        \]
        The dimensions are finite,  since a dimension only depends on the underlying topological space, and the map on topological spaces $f:|X|\to |Y|$ is isomorphic to a corresponding morphism of finitely presented schemes over $\spec\resfield$ by \cref{perfschemes:prop:existsfinitetypemodel}.
        
        

\end{para}

\begin{definition}\label{dimthry:def:equidimensionaloffinitetype}
    Let $f\colon X \to Y$ be a morphism of perfect algebraic spaces that are perfectly finitely presented over \( \spec\resfield \).
    \begin{mylist}
        \item We say that $f$ is \emph{weakly equidimensional}, if the dimension function $\dimfunc_f$ is locally constant on $|X|$.
        \item We say that $f$ is \emph{equidimensional}, if it is weakly equidimensional and for every $x \in X$ we have
        \[
        \dim_x f^{-1}(f(x)) = \dimfunc_f(x).
        \]
        \item We say that $f$ is \emph{uo-equidimensional},  if it is universally open and equidimensional.
    \end{mylist}
\end{definition}

\begin{example}
    A locally closed embedding $X \hookrightarrow Y$ is said to be \emph{of pure codimension $d$}, if it is weakly equidimensional of relative dimension $-d$. Since a locally closed embedding is injective on underlying topological spaces, it can not be equidimensional unless $d = 0$.
\end{example}

        
        
        

\begin{example}\label{dimthry:prop:flatisuoequidim}
    Let $f: X \to Y$ be a flat morphism between irreducible schemes perfectly finitely presented over $\spec \resfield$. Then $f$ is uo-equidimensional.
\end{example}
\begin{proof}
    First, we show that $f$ is open. By taking a model finitely presented over $\spec \resfield$ (see \cref{perfschemes:prop:existsfinitetypemodel}), the image under $f$ of any open subset is constructible. Hence, to show $f$ is open, it suffices to use the going down theorem, which holds for general flat morphisms. Thus, $f$ is also universally open, since flatness is preserved under base change. Next, $f$ is weakly equidimensional by irreducibility. Hence, $f$ is uo-equidimensional, since universally open and weakly equidimensional imply uo-equidimensional by the same argument as in \cite[Corollary 2.1.5.]{bouthier2022perverse}.
\end{proof}

\begin{definition}\label{dimthry:def:IPP}
    Let $\varcat{E} \subset \mor(\peraff^{\pfp})$ be a class of morphisms between perfect affine schemes perfectly finitely presented over $\spec \resfield$, such that:
    \begin{mylist}
        \item $\varcat{E}$ contains all perfectly smooth morphisms;
        \item $\varcat{E}$ is closed under composition;
        \item $\varcat{E}$ is closed under pullbacks along perfectly smooth morphisms.
    \end{mylist}
    We say that $\varcat{E}$ satisfies the \emph{independence of presentation property (IPP)}, if for any two placid presentations $X \simeq \lim X_\alpha$ and $X \simeq \lim X'_\beta$, and for any index $\alpha$, there exists an index $\beta$ and a morphism $f_{\beta,\alpha} \colon X'_\beta \to X_\alpha$ under $X$ that belongs to $\varcat{E}$.
\end{definition}

\begin{proposition}\label{dimthry:prop:uoequidimsatisfyIPP}
    The classes introduced in \cref{dimthry:def:equidimensionaloffinitetype}—namely, weakly equidimensional, equidimensional, and uo-equidimensional morphisms—as well as the class of universally open morphisms, satisfy the IPP property.
\end{proposition}

\begin{proof}
    These classes satisfy conditions (a)–(c) of \cref{dimthry:def:IPP} by similar arguments to their analogs for non-perfect schemes (see \cite[A.5.2(a)]{bouthier2022perverse}). The verification of the IPP property then follows as in the proof of \cite[A.5.4]{bouthier2022perverse}.\footnote{Note that the argument for the IPP property in \cite[A.5.2(b)] {bouthier2022perverse} fails in this case because (at least a priori) perfectly smooth morphisms do not satisfy the IPP property.}
\end{proof}

\begin{definition}\label{dimthry:def:equidimforplacidpresented} (see \cite[Construction A.3.1.]{bouthier2022perverse})  Let $f \colon X \to Y$ be a morphism between perfect schemes admitting perfectly placid presentations.  
We say that $f$ is \emph{universally open}, \emph{weakly equidimensional}, \emph{equidimensional}, or \emph{uo-equidimensional}, if for every pair of placid presentations $X \simeq \lim X_\alpha$ and $Y \simeq \lim Y_\beta$, and for every $\beta$, there exist  $\alpha$ and a commutative diagram
\[
\begin{tikzcd}
	X \arrow[r,"f"] \arrow[d] & Y \arrow[d] \\
	X_\alpha \arrow[r,"f_{\alpha\beta}"] & Y_\beta, 
\end{tikzcd}
\]
where the vertical arrows are the canonical projections, and $f_{\alpha\beta}$ has the corresponding property (universally open, weakly equidimensional, etc.) in the sense of \cref{dimthry:def:equidimensionaloffinitetype}.
\end{definition}



\begin{definition}\label{dimthry:def:defofequidimforplacidstacks}
(see \cite[Construction A.2.4.]{bouthier2022perverse})
Let \( f \colon \varstack{X} \to \varstack{Y} \) be a morphism between perfectly placid \( \infty \)-stacks.  
We say that \( f \) is \emph{universally open}, \emph{weakly equidimensional}, \emph{equidimensional}, or \emph{uo-equidimensional}, if for every perfectly smooth morphism \( Y \to \varstack{Y} \) with \( Y \) admitting a perfectly placid presentation, there exists a perfectly smooth cover \( \{X_\alpha \to \varstack{X} \times_{\varstack{Y}} Y\}_\alpha \) such that:
\begin{mylist}
    \item each \( X_\alpha \) admits a perfectly placid presentation, and
    \item the composition \( X_\alpha \to \varstack{X} \times_{\varstack{Y}} Y \to Y \) is universally open, weakly equidimensional, equidimensional, or uo-equidimensional, respectively, in the sense of \cref{dimthry:def:equidimforplacidpresented}.
\end{mylist}
\end{definition}

\begin{lemma}\label{dimthry:prop:pfpfbetweenplacidpresentedisequidimiffbasechangeofpfp}
Let \( f \colon X \to Y \) be a perfectly finitely presented morphism between perfect affine schemes, both admitting perfectly placid presentations. Then \( f \) is universally open / weakly equidimensional / equidimensional / uo-equidimensional if and only if the following holds:

There exists a perfectly pro-smooth morphism \( Y \to Y_0 \), and a morphism \( f_0 \colon X_0 \to Y_0 \) which is universally open / weakly equidimensional / equidimensional / uo-equidimensional, such that:
\begin{mylist}
    \item \( X_0 \) and \( Y_0 \) are perfect schemes, perfectly finitely presented over \( \spec\resfield \), and
    \item there is an isomorphism \( f \simeq f_0 \times_{Y_0} Y \).
\end{mylist}
\end{lemma}


\begin{proof}
    'Only if' direction is shown in \cite[A.4.2(b)]{bouthier2022perverse}.

    Let us prove the 'if' direction.  We use (a corrected version of) \cite[A.4.11(c)]{bouthier2022perverse}\footnote{Note that in \cite[A.4.11(c)]{bouthier2022perverse} the statement is mistakenly tautological. The statement that is actually proven in \cite[A.4.11(c)]{bouthier2022perverse} is: For a \( \varcat{B} \)-representable morphism \( f\colon x \to y \) such that \( x,y \in \ob_0(\varstack{A}) \), we have that \( f \in \mor_0(\varstack{A}) \) if and only if there exists a \( \mor_0(\varcat{B}) \)-presentation \( y \simeq \lim_\alpha y_\alpha \), an index \( \alpha \), and a morphism \( f_\alpha\colon x_\alpha \to y_\alpha \) in \( \mor_0^+(\varcat{B}) \) such that \( f \simeq f_\alpha \times_{y_\alpha} y \).}. 
    
    Thus, we need to show that \( f \) is \( \peraff^{\pfp} \)-representable in the sense of \cite[A.4.9]{bouthier2022perverse} (it is a base change of morphisms between perfect affine schemes perfectly finitely presented over \( \spec\resfield \)).
    
    Let \( f\colon X \to Y \) be a perfectly finitely presented morphism between perfect affine schemes admitting perfectly placid presentations. Let \( f'\colon X' \to Y \) be a finitely presented morphism such that \( f \simeq (f')_{\perf} \) (see \cref{perfschemes:prop:existsfinitetypemodel}).

    Let \( Y \simeq \lim Y_\alpha \) be some perfectly placid presentation. By \cite[\href{https://stacks.math.columbia.edu/tag/01ZM}{Tag 01ZM}]{stacks-project}, there exists some \( \alpha \) and some \( f'_0\colon X'_0 \to Y_\alpha \) such that \( f' \simeq f'_0 \times_{Y_\alpha} Y \). Applying \( (-)_{\perf} \), which preserves all limits, we get that \( f \simeq (f'_0)_{\perf} \times_{Y_\alpha} Y \). 
\end{proof}

\begin{para}{\bf Equidimensional morphisms of relative dimension $d$ between perfectly finitely presented perfect affine schemes.}
    \begin{mylist}
        \item Let \( f \colon X \to Y \) be a universally open / weakly equidimensional / equidimensional / uo-equidimensional morphism between perfect affine schemes, perfectly finitely presented over \( \spec\resfield \). We say that \( f \) is \emph{of relative dimension \( d \)}, if the dimension function \( \dimfunc_f \) is constant equal to \( d \).
        
        \item \label{dimthry:prop:equidimofrelativedimensionpreservedbypullback} 
        Universally open / weakly equidimensional / equidimensional / uo-equidimensional morphisms of relative dimension \( d \) are preserved under base change along open morphisms, and the property is reflected under base change along surjective open morphisms.
    \end{mylist}
\end{para}

\begin{proof}
    Part (b) follows immediately from the analogue in the perfect setting of \cite[Lemma 2.2.4. (a)]{bouthier2022perverse}.
\end{proof}

\begin{para}\label{dimthry:prop:constantrelativedimension}{\bf Pfp-representable equidimensional morphisms of relative dimension $d$ between perfect affine schemes.}
    \begin{mylist}
        \item Let $f \colon X \to Y$ be a perfectly finitely presented morphism between perfect affine schemes admitting perfectly placid presentations. We say that $f$ is universally open / weakly equidimensional / equidimensional / uo-equidimensional \emph{of relative dimension $d$}, if there exists a morphism $f_0$ as in \cref{dimthry:prop:pfpfbetweenplacidpresentedisequidimiffbasechangeofpfp} such that $\dimfunc_{f_0}$ is constant equal to $d$.
        
        \item\label{dimthry:prop:purecodimschemes} A weakly equidimensional, perfectly finitely presented, locally closed embedding of relative dimension $-d$ will be called \emph{of pure codimension $d$}.
    \end{mylist}
\end{para}

\begin{proposition}\label{dimthry:prop:compositionofequidim} 
    The composition of a weakly equidimensional morphism of relative dimension $d$ between perfect affine schemes admitting perfectly placid presentations with a weakly equidimensional morphism of relative dimension $d'$ between perfect affine schemes admitting perfectly placid presentations is weakly equidimensional morphism of relative dimension $d + d'$.
\end{proposition}

\begin{proof}
    Let \( f \colon X \to Y \) and \( g \colon Y \to Z \) be morphisms such that
    \[
    f \simeq f_0 \times_{Y_0} Y \quad \text{and} \quad g \simeq g_0 \times_{Z_0} Z,
    \]
    where \( f_0 \colon X_0 \to Y_0 \) and \( g_0 \colon Y'_0 \to Z_0 \) are weakly equidimensional morphisms between perfect schemes, perfectly finitely presented over \( \spec\resfield \), and of pure codimensions \( d \) and \( d' \), respectively, and \( Z \to Z_0 \) and \( Y \to Y_0 \) are strongly pro-perfectly smooth. Note that \( Y_0 \) and \( Y'_0 \) need not coincide.

    Choose a presentation \( Z \simeq \lim Z_\alpha \) of \( Z\to Z_0 \) as a strongly pro-perfectly smooth morphism,  and define
    \(
    Y'_\alpha := Y'_0 \times_{Z_0} Z_\alpha.
    \)
    After possibly enlarging \( \alpha \), we may assume \( Y'_\alpha \to Y_0 \) is a uo-equidimensional morphism. Define
    \[
    g_\alpha := g_0 \times_{Z_0} Z_\alpha \quad \text{and} \quad f_\alpha := f_0 \times_{Y_0} Y'_\alpha.
    \]
    Then \( f \simeq f_\alpha \times_{Y'_\alpha} Y \) and \( g \simeq g_\alpha \times_{Z_\alpha} Z \), so the composition satisfies
    \[
    g \circ f \simeq (g_\alpha \circ f_\alpha) \times_{Z_\alpha} Z.
    \]
    
    The statement for perfectly finitely presented perfect schemes  is analogous to the case of finitely presented schemes, see \cite[Lemma 2.1.6(b)]{bouthier2022perverse}. We find that \( g_\alpha \circ f_\alpha \) is of pure codimension \( d + d' \), and the result follows.
\end{proof}

\begin{para}\label{dimthry:prop:constantrelativedimensionbetweenplacid} {\bf Pfp-representable equidimensional morphisms of relative dimension $d$ between perfectly placid perfect $\infty$-stacks.}
    Let \( f\colon \varstack{X} \to \varstack{Y} \) be a pfp-representable morphism (see \cref{perfstacks:def:pfprepresentable}) between perfectly placid $\infty$-stacks. We say that \( f \) is weakly equidimensional / equidimensional / uo-equidimensional \emph{of relative dimension \( d \)},  if there exists a perfectly smooth covering \( \coprod Y_\alpha \to \varstack{Y} \), where each \( Y_\alpha \) is an affine scheme admitting a perfectly placid presentation, and affine \'etale covers \( \coprod Z_{\alpha,\beta} \to \coprod \varstack{X} \times_{\varstack{Y}} Y_\alpha \), such that the composition
    \[
    \coprod Z_{\alpha,\beta} \to \coprod \varstack{X} \times_{\varstack{Y}} Y_\alpha \to \coprod Y_\alpha
    \]
    is weakly equidimensional / equidimensional / uo-equidimensional of relative dimension \( d \).
\end{para}

\begin{definition}
    \label{dimthry:prop:purecodimstacks} A weakly equidimensional, perfectly finitely presented, locally closed embedding $X\into Y$ of relative dimension $-d$ will be called \emph{of pure codimension $d$}. We will then write \[d=\codim_Y(X).\]
\end{definition}

\begin{remark}
    The notions of pfp-representable morphisms with constant relative dimension given in \cref{dimthry:prop:constantrelativedimension} and \cref{dimthry:prop:constantrelativedimensionbetweenplacid} are well-defined in the sense that the relative dimension \( d \) is uniquely determined.
    This can be shown using the methods of \cite[\S 2.2]{bouthier2022perverse}, but we omit the proof for brevity.
\end{remark}

\begin{remark}
    \cref{dimtry:prop:uo_equidim_cover_by_smooth_imply_uo_equidim} and  \cref{dimthry:prop:pfpfbetweenplacidpresentedcorollaries} correspond to \cite[Corollary 1.1.7]{bouthier2022perverse} and \cite[Lemma 2.2.10]{bouthier2022perverse}, respectively. The statements here are slightly more involved, as perfectly smooth morphisms do not satisfy the IPP property, unlike in the non-perfect setting.
\end{remark}

\begin{proposition}\label{dimtry:prop:uo_equidim_cover_by_smooth_imply_uo_equidim}
    Let \( f \colon X \onto Y \) be a perfectly finitely presented, perfectly smooth, surjective morphism of affine schemes admitting perfectly placid presentations. If \( X \) is strongly pro-perfectly smooth, then \( Y \to * \) is uo-equidimensional.
\end{proposition}

\begin{proof}
    Take a perfectly finitely presented, uo-equidimensional model \( f_0 \colon X_0 \to Y_0 \) such that $Y\to Y_0$ is strongly pro-perfectly smooth and \( f \simeq f_0 \times_{Y_0} Y \) (see \cref{dimthry:prop:pfpfbetweenplacidpresentedisequidimiffbasechangeofpfp}). 

    Since $Y\to Y_0$ is strongly pro-perfectly smooth, it is uo-equidimensional. Thus, it is enough to show that $Y_0\to *$ is uo-equidimensional. The morphism $Y_0\to *$ is always universally open (see \cite[\href{https://stacks.math.columbia.edu/tag/0383}{Lemma 0383}]{stacks-project}), so it suffices to show that it is equidimensional.
    
    We claim that it is enough to show that \( X_0 \) is equidimensional. Note that \(X_0 \to Y_0\) is equidimensional, and the dimension function compose additively (this follows from the finitely presented case \cite[Lemma 2.1.6(b)]{bouthier2022perverse} by taking a finitely presented model). Thus, if $X_0$ is equidimensional, then $Y_0$ is equidimensional as well.

    Let \( Y \simeq\lim_{\alpha} Y_\alpha \) be a presentation of \( Y \to Y_0 \) as a strongly pro-perfectly smooth morphism. Then  $X$ has a perfectly placid presentation 
    \( X \simeq\lim_{\alpha} X_\alpha \) with \( X_\alpha:= X_0 \times_{Y_0} Y_\alpha \) and another one \( X \simeq \lim_{\beta} X_\beta' \), in which each \( X_\beta' \) is perfectly smooth. By the IPP property for uo-equidimensional morphisms (see \cref{dimthry:prop:uoequidimsatisfyIPP}), there exists a uo-equidimensional morphism \( X_\alpha \to X_\beta' \). Since \( X_\beta' \) is equidimensional (see \cref{dimthry:prop:flatisuoequidim}), so is \( X_\alpha \). Finally, replacing \( f_0 \) with \( f_\alpha \simeq f_0 \times_{Y_0} Y_\alpha \), we obtain that $X_0$ is equidimensional, and we are done.   
\end{proof}

\begin{proposition}\label{dimthry:prop:pfpfbetweenplacidpresentedcorollaries}
    Let \( f:Z\to Y \) be a perfectly finitely presented morphism of connected perfect affine schemes admitting placid presentations. If \( Y \to * \) and \( Z \to * \) are equidimensional, then \( f \) is weakly equidimensional of relative dimension $d$ for some $d\in \ints$.
\end{proposition}

\begin{proof}
    By \cref{dimthry:prop:pfpfbetweenplacidpresentedisequidimiffbasechangeofpfp}, there exists a morphism \( f_1 \colon Y_1 \to Z_1 \) between perfect affine schemes, perfectly finitely presented over \( \spec \resfield \), such that \( f \simeq f_1 \times_{Z_1} Z \), with \( Z \to Z_1 \) a strongly pro-perfectly smooth morphism. By the same argument as \cref{dimtry:prop:uo_equidim_cover_by_smooth_imply_uo_equidim}, we may assume \( Y_1, Z_1 \) are equidimensional. Since $Y,Z$ are connected, we may assume $Y_1,Z_1$ are connected by passing to their respective connected components.

    Thus, \( f_1 \) is a morphism between equidimensional, connected perfect affine schemes, perfectly finitely presented over \( \spec \resfield \). Hence, \( f_1 \) is weakly equidimensional and of relative dimension $d$ for some $d$. Since \( f \) is a base change of a weakly equidimensional morphism of relative dimension $d$ along a strongly pro-perfectly smooth morphism, it is itself weakly equidimensional of relative dimension $d$.
\end{proof}

\section{Perfectly placid stratifications and semi-small morphisms}\label{stratification}

In this section, we define a notion of perfectly placid stratification—an infinite stratification whose strata are perfectly placid perfect \(\infty\)-substacks. We follow \cite[\S 2.4]{bouthier2022perverse} closely.

Note that in our setting, \cref{prop:comp_of_comp} has a slightly nicer formulation than the corresponding result of \cite[\S 2.4.1(c)]{bouthier2022perverse} (compare \cref{perfstacks:remark:noneedtowritered}).

\begin{para}{\bf Complementary substacks.}\label{stratification:stratification}
    Let \( \varstack{X} \) be a perfect \( \infty \)-stack, and let \( \varstack{Y} \into \varstack{X} \) be a (pfp-)locally closed embedding.
    \begin{mylist}
        \item We define \( \varstack{X} \smallsetminus \varstack{Y} \) to be the perfect substack of \( \varstack{X} \) such that, for any test object \( U \), the anima \( (\varstack{X} \smallsetminus \varstack{Y})(U) \subset \varstack{X}(U) \) consists of morphisms \( U\to \varstack{X} \) such that \( U \times_{\varstack{X}} \varstack{Y} = \emptyset \).
        
        \item If \( \varstack{Y} \into \varstack{X} \) is a (pfp-)closed embedding, then \( \varstack{X} \smallsetminus \varstack{Y} \into \varstack{X} \) is a (pfp-)open embedding. Similarly, if \( \varstack{Y} \into \varstack{X} \) is a (pfp-)open embedding, then \( \varstack{X} \smallsetminus \varstack{Y} \into \varstack{X} \) is a (pfp-)closed embedding. This follows from the case where \( \varstack{X} \) is a perfect affine scheme, perfectly finitely presented over \( \spec \resfield \).

        \item\label{prop:comp_of_comp} We have that $\varstack{X}\smallsetminus (\varstack{X}\smallsetminus \varstack{Y}) \simeq \varstack{Y}$.
    \end{mylist}
\end{para}

\begin{para}{\bf Notations and definitions.}
    \begin{mylist}
        \item\label{stratification:def:collectionoflocallyclosed}
        Let \( \explicitset{\varstack{X_\alpha} \into \varstack{X}}_{\alpha \in I} \) be a collection of pfp-locally closed embeddings of perfect \( \infty \)-stacks such that \( \varstack{X}_{\alpha} \times_{\varstack{X}} \varstack{X}_{\beta} = \emptyset \) for \( \alpha \ne \beta \).

        \item We say that a locally closed embedding \( \varstack{X}' \into \varstack{X} \) is {\em adapted}, if for each \( \alpha \in I \), we have either \( \varstack{X}_\alpha \times_{\varstack{X}} \varstack{X}' = \emptyset \), or \( \varstack{X}_\alpha\into \varstack{X} \) factors through \( \varstack{X}' \).

        \item For any locally closed embedding \( \varstack{X}' \into \varstack{X} \), we let \( I_{\varstack{X}'} \subset I \) denote the set of indices \( \alpha \) for which \( \varstack{X}_\alpha\times_{\varstack{X}}\varstack{X}'\neq \emptyset \).
    \end{mylist}
\end{para}

\begin{para}{\bf Constructible stratification.}
    Let \( \explicitset{\varstack{X_\alpha} \into \varstack{X}}_{\alpha \in I} \) be as in \cref{stratification:def:collectionoflocallyclosed}.
    \begin{mylist}
        \item\label{stratification:def:finitestrat}
        We say that \( \explicitset{\varstack{X_\alpha} \into \varstack{X}}_{\alpha \in I} \) is a \emph{finite constructible stratification}, if \( I \) is finite and there exists a sequence of open embeddings
        \[
        \emptyset = \varstack{X}_0 \subset \varstack{X}_1 \subset \dots \subset \varstack{X}_n = \varstack{X}
        \]
        and an ordering \( \alpha_1, \ldots, \alpha_n \) of \( I \) such that each embedding \( \varstack{X}_{\alpha_i}\into \varstack{X} \) induces an isomorphism \( \varstack{X}_{\alpha_i} \xto{\sim} \varstack{X}_{i} \smallsetminus \varstack{X}_{i-1} \).

        \item\label{stratification:def:boundedstrat}
        We say that \( \explicitset{\varstack{X_\alpha} \into \varstack{X}}_{\alpha \in I} \) is a \emph{bounded constructible stratification}, if there exists a filtered diagram \( (\varstack{X}_{j})_{j \in J} \) of adapted pfp-open embeddings such that \( \varstack{X} \simeq \colim \varstack{X}_j \), and for all \( j \), the collection \( \explicitset{\varstack{X}_{\alpha} \into \varstack{X}_j}_{\alpha \in I_{\varstack{X}_j}} \) is a finite constructible stratification.
    \end{mylist}
\end{para}

The following properties of constructible stratifications are  straightforward and follow by arguments entirely analogous to those in \cite[\S 2.4]{bouthier2022perverse}: 

\begin{para}{\bf Pullback of constructible stratifications.}
    \label{stratification:prop:pullbackofstrat}
    Let \( f \colon \varstack{X} \to \varstack{Y} \) be a morphism of perfect \( \infty \)-stacks, and let \(  \explicitset{\varstack{Y}_{\alpha} \into \varstack{Y}}_{\alpha \in I} \) be a (finite/bounded) constructible stratification. Then the pullback
    \[
    \explicitset{f^{-1}(\varstack{Y}_{\alpha}) \into \varstack{X}}_{\alpha \in I}
    \]
    is a (finite/bounded) constructible stratification.
\end{para}

\begin{para}{\bf Perfectly placidly stratified \(\infty\)-stacks.}
    A \emph{perfectly placid stratification} is a bounded constructible stratification \( \explicitset{\varstack{X}_{\alpha} \into \varstack{X}}_{\alpha \in I} \) in which each stratum \( \varstack{X}_{\alpha} \) is perfectly placid. A \emph{perfectly placidly stratified \( \infty \)-stack} is an \( \infty \)-stack \( \varstack{X} \) equipped with a perfectly placid stratification \( \explicitset{\varstack{X}_{\alpha} \into \varstack{X}}_{\alpha \in I} \).
\end{para}

\begin{para}\label{stratification:def:semismall}{\bf (Semi-)small morphisms.}
    \begin{mylist}
        \item Let \( f \colon \varstack{X} \to \varstack{Y} \) be a morphism of perfect \( \infty \)-stacks, and let \( \explicitset{\varstack{Y}_{\alpha}}_{\alpha \in I} \) be a perfectly placid stratification. Let \( f_{\alpha} \colon \varstack{X}_{\alpha} \to \varstack{Y}_{\alpha} \) denote the restriction of \( f \) to the preimage of each stratum.

        \item Assume that \( \varstack{X} \) is perfectly placid, each \( \varstack{X}_{\alpha} \subset \varstack{X} \) is of pure codimension \( b_{\alpha} \), and each \( f_{\alpha} \colon \varstack{X}_{\alpha} \to \varstack{Y}_{\alpha} \) is locally pfp-representable and equidimensional of constant relative dimension \( \delta_{\alpha} \).

        \item If \( \delta_{\alpha} \leq b_{\alpha} \) for all \( \alpha \in I \), we say that \( f \) is \emph{semi-small}.

        \item Let \( \varstack{U} \into \varstack{Y} \) be a \( \explicitset{\varstack{Y}_{\alpha}}_{\alpha \in I} \)-adapted open substack. We say that a semi-small morphism \( f \) is \emph{\( \varstack{U} \)-small},  if for all \( \alpha \in I \smallsetminus I_{\varstack{U}} \), we have \( \delta_{\alpha} < b_{\alpha} \).
    \end{mylist}
\end{para}

\section{The GKM stratification}\label{GKMstrat}
In this section we discuss the basic definitions, geometry and dimension theory of a stratification on $\compactelements$.

\subsection*{Some generalities.}

\begin{para}{\bf The Witt vector functor.}\label{wittloopandarc:def:wittvecs}
\begin{mylist}
    \item Let \( \wittzp:{\mathbb{F}_p}\text{-}\Alg^{\perf}\to {\ints_p}\text{-}\Alg \) denote the usual Witt vector functor.
    
    \item Let \( \resfield \) be an algebraically closed field of characteristic \( p \), and let \( K \) be a finite extension of \( \wittzp(\resfield)[1/p] \).
    
    \item Let \( \basicints \) be the integral closure of \( \wittzp(\resfield) \) in \( K \), and let $\uniformizer\in \basicints$ be a uniformizer. 
    
    \item For every perfect $\resfield$-algebra $R$,  define \( \witt(R) := \wittzp(R) \otimes_{\wittzp(\resfield)} \basicints \), the \( \basicints \)-Witt vector construction. We view this as a functor \( \resfield\text{-}\Alg^{\perf} \to \basicints\text{-}\Alg \).
    
    \item For every perfect $\resfield$-algebra $R$, let \( \wittn(R) := \witt(R)/(\uniformizer^n) \).
    
    \item Similarly, we have the functor \( \resfield\text{-}\Alg^{\perf} \to \basicextension\text{-}\Alg \), sending \( R \mapsto \witt(R)\left[1/p\right] \).
\end{mylist}
\end{para}

\begin{para}\label{wittloopandarc:def:loopandarc}{\bf Arc and loop functors.}
    \begin{mylist}
        \item The functors \( \witt(-) \), \( \wittn(-) \), and \( \witt(-)[1/p] \) from \cref{wittloopandarc:def:wittvecs} induce natural functors
        \( \arcspace, \arcspace_n \colon \prestacks_{\basicints} \to \prestacks_{\perf,\resfield} \), and
        \( \loopspace \colon \prestacks_{\basicextension} \to \prestacks_{\perf,\resfield} \). E.g. for $\varstack{X}\in \prestacks_{\basicextension}$ and $R\in \resfield\text{-}\Alg$, we set \[\loopspace (\varstack{X})(R):= \varstack{X}(\witt(R)[1/p]).\]

        \item If \( X \) is an affine scheme, finitely presented over \( \basicints \), then \( \arcspace_n( X )\) is a perfect affine scheme perfectly finitely presented over $\resfield$, and \( \arcspace( X) \) is a perfect affine scheme \cite[\S~1.1.1]{zhu2017affine}.
    \end{mylist}
\end{para}

\begin{example}\label{wittloopandarc:example:1daffinespace}
    Let \( X = \affinespace^1_{\basicints} \). Then \( \arcspace_n(X) \simeq \affinespace^n_{\perf} \), with the identification given by the morphism \( \affinespace^n_{\perf} \to \arcspace_n(X) \) defined on \( R \)-points by
    \[
    (a_0, \ldots, a_{n-1}) \mapsto \sum_{i=0}^{n-1} [a_i] \uniformizer^i,
    \]
    where \( [a_i] \) denotes the Teichmüller lift.

    We also have \( \arcspace (X )\simeq \lim_n \affinespace^n_{\perf} \simeq \spec\left( \resfield[a_i]_{i \geq 0} \right)_{\perf} \).
\end{example}

\begin{proposition}\label{wittloopandarc:properties:arpreserveetale}{\bf \'Etale morphism are preserved by $\arcspace$.}
    Let \( f \colon X \to Y \) be an \'etale morphism of schemes, finitely presented over \( \basicints \). Then we have a Cartesian square:
    \[\begin{tikzcd}
	{\arcspace (X)} & {\arcspace (Y)} \\
	{\arcspace_0( X)} & {\arcspace_0( Y).}
	\arrow["\arcspace (f)", from=1-1, to=1-2]
	\arrow[from=1-1, to=2-1]
	\arrow[from=1-2, to=2-2]
	\arrow["\arcspace_0( f)",from=2-1, to=2-2]
    \end{tikzcd}\]
    In particular, the morphism $\arcspace (f):\arcspace (X)\to \arcspace (Y)$ is \'etale.
\end{proposition}

\begin{proof}
    Since $f$ is \'etale, $f\times_{\spec \basicints}\resfield: X\times_{\spec \basicints}\resfield \to Y\times_{\spec \basicints}\resfield$ is \'etale. Since perfection preserves \'etale (see \cref{perfschemes:prop:propertiesofperfschemespreserved}) we get that $\arcspace_0(f)$ is \'etale.
    The rest of the proof follows the same as in \cite[Lemma 3.1.3]{bouthier2022perverse}.
\end{proof}

\begin{example}\label{wittloopandarc:example:loopofGm}
    There is a natural isomorphism of group schemes\footnote{In the equal characteristic case, this holds only after reduction. In our setting, there is no such notion; see \cref{perfstacks:remark:noneedtowritered}.}
    \[
    \loopspace (\multiplicativegroup) \simeq \ints \times \arcspace (\multiplicativegroup).
    \]
\end{example}

\begin{proof}
    Consider the morphism \( \ints \times \arcspace (\multiplicativegroup) \to \loopspace (\multiplicativegroup) \) given by
    \[
    (n, f) \mapsto \uniformizer^n f.
    \]
    We want to show that this morphism is an isomorphism.

    It is a closed embedding: it is injective, and its restriction to the connected component of the identity is the canonical closed embedding \( \arcspace (\multiplicativegroup) \into \loopspace (\multiplicativegroup) \). Thus, by \cref{perfstacks:prop:graddream}, it suffices to check that the map is an isomorphism on \( L \)-points for any algebraically closed field \( F / \resfield \). This can be readily verified.
\end{proof}

\begin{para}{\bf Notations.}
\begin{mylist}

    \item\label{wittloopandarc:def:constants} Let \( \thegroup \) be a connected reductive group over \( \basicints \) with Lie algebra \( \thelie \). Let \( g \mapsto \theadjoint_g \) denote the adjoint action of \( \thegroup \) on \( \thelie \). Let \( T \subset B \subset G\) be a torus contained in a Borel subgroup. Let \( \theweylgroup = N_G(\thetorus)/\thetorus \) be the Weyl group of $G$, let \( X_*(\thetorus) \) denote the cocharacter lattice, and let \( \theextendedweylgroup:=W\ltimes X_*(T) \) be the extended affine Weyl group of $G$. Let \( R \subset X_*(T) \) denote the set of roots. Let \( \thetoruslie \) and \( \theborellie \) be the Lie algebras of \( \thetorus \) and \( \theborel \), respectively. Set \( r:= \dim \thetoruslie \), the rank of the group.
    \item\label{wittloopandarc:def:chevalley} Let \( \chevalley = \thetoruslie // \theweylgroup := \spec \basicints[\thetoruslie]^{\theweylgroup} \). Recall that \( \chevalley \simeq \thelie // \thegroup \). We have canonical projections \( \chi \colon \thelie \to \chevalley \) and \( \pi \colon \thetoruslie \to \chevalley \). Recall that \( \pi \) is a finite, flat, surjective morphism. Recall that $\chevalley\simeq \affinespace_\basicints^r$.    \item\label{wittloopandarc:def:discriminant} Let \( \thediscriminant \in \basicints[\thetoruslie] \) denote the discriminant function \( \thediscriminant = \prod_{\alpha \in \theroots} \operatorname{d}\alpha \). We have \( \thediscriminant \in \basicints[\chevalley] \subset \basicints[\thetoruslie] \). Define \( \chevalley^{\regsemisimple} \) to be the locus where \( \thediscriminant \ne 0 \). Set \( \thelie^{\regsemisimple} := \chi^{-1}(\chevalley^{\regsemisimple}) \) and \( \thetoruslie^{\regsemisimple} := \pi^{-1}(\chevalley^{\regsemisimple}) \).

    \item\label{wittloopandarc:def:arcofchi} The morphism \( \chi \) induces morphisms \( \arcspace (\chi) \colon \arcspace (\thelie) \to \arcspace (\chevalley) \) and \( \arcspace_n( \chi )\colon \arcspace_n( \thelie )\to \arcspace_n( \chevalley) \) of (truncated) arc spaces.
\end{mylist}
\end{para}

\begin{para}{\bf The Iwahori group and its Lie algebra.}
    \begin{mylist}
        \item\label{wittloopandarc:def:iwahori} Let \( \iwahory := \arcspace (\thegroup) \times_{\arcspace_0(\thegroup)} \arcspace_0( \theborel )\) denote a fixed Iwahori subgroup. Let \( \iwahory_n := \arcspace_n( \thegroup) \times_{\arcspace_0( \thegroup)} \arcspace_0( \theborel) \) denote the truncated Iwahori subgroup.

        \item Let \( \liewahory := \arcspace (\thelie) \times_{\arcspace_0(\thelie)} \arcspace_0( \theborel) \) and \( \liewahory_n := \arcspace_n(\thelie) \times_{\arcspace_0( \thelie)} \arcspace_0( \theborel) \) denote the corresponding Iwahori Lie algebras.

        \item\label{wittloopandarc:def:arcofnu} The morphism \( \chi \) induces morphisms \( \arcspace (\nu) \colon  \liewahory \to \arcspace (\chevalley) \) and \( \arcspace_n(\nu) \colon  \liewahory_n \to \arcspace_n( \chevalley) \) of (tructated) arc spaces.
    \end{mylist}
\end{para}

\begin{para}{\bf Stratification of \( \arcspace (\chevalley) \).}
    \begin{mylist}
        \item Set $\arcspace (\chevalley)_\bullet = \arcspace (\chevalley)_{\thediscriminant\neq 0}$.
        By an analogue of \cite[\S 3.2.2.]{bouthier2022perverse}, 
        there is a canonical bounded constructible stratification \( \arcspace(\chevalley)_{(\thediscriminant; n)} \) on \( \arcspace (\chevalley)_\bullet \), given by the valuation of the function \( \thediscriminant \) defined in \cref{wittloopandarc:def:discriminant}.

        \item Similarly, there is a canonical bounded constructible stratification \( \arcspace(\thetoruslie)_{(\thediscriminant; n)} \) on \( \arcspace (\thetoruslie_\bullet )\), also given by the valuation of \( \thediscriminant \).

        \item Since \( \thediscriminant \colon \thetoruslie \to \affinespace^1 \) factors through \( \pi \colon \thetoruslie \to \chevalley \) (see \cref{wittloopandarc:def:discriminant}), we have
        \[
        \pi^{-1} \big( \arcspace(\chevalley)_{(\thediscriminant; n)} \big) = \arcspace(\thetoruslie)_{(\thediscriminant; n)}.
        \]
    \end{mylist}
\end{para}

\begin{para}{\bf Ramified extension.}
    \begin{mylist}
        \item We let $\basicextension' := \basicextension[\uniformizer^{1/a}]$ and $ \basicints' := \basicints[\uniformizer^{1/a}]$, where $a=|\theweylgroup|$.
        \item Let $\thetoruslie':= \weilres_{\basicints}^{\basicints'}(\thetoruslie\times_{\spec\basicints} \spec \basicints')$ and $\chevalley':= \weilres_{\basicints}^{\basicints'}(\chevalley\times_{\spec\basicints} \spec \basicints')$, where $\weilres_{\basicints}^{\basicints'}$ stands for the Weil restriction.
        \item Then, for a perfect $\resfield$-algebra $R$,  we have  
        \[\arcspace(\thetoruslie')(R)=\thetoruslie(\witt(R) \otimes_{\basicints} \basicints')\text{ and } \arcspace(\chevalley')(R)=\chevalley(\witt(R) \otimes_{\basicints} \basicints').\]
    \end{mylist}
\end{para}

\begin{para}{\bf The stratification on the twisted tori.}\label{stratoftorus:twistedstrat}
    \begin{mylist}
        \item For $w \in W$, we define \( \thetoruslie_{w} \) as in \cite[\S 3.3.3.(b)]{bouthier2022perverse}.
        We set $\thetoruslie_{w}:= (\thetoruslie')^{\sigma w}$, where $\sigma:\basicints'\to \basicints'$ is a fixed generator of the Galois group of $\basicints'$ over $\basicints$.
        \item There is a non-canonical isomorphism \( \thetoruslie_w \simeq \affinespace^r_\basicints \). 
        \item For a function $\rootfunc:\theroots\to \tfrac{1}{a}\ints_{\geq0}$ and $w\in W$, we define a subscheme \( \thetoruslie_{w,\rootfunc}\subset\arcspace(\thetoruslie_w)\) as in 
          \cite[\S 3.3.3.(c)]{bouthier2022perverse}.  
          
        \item Some of the strata \( \thetoruslie_{w,\rootfunc} \) are empty; we will not concern ourselves with identifying these, as this is not needed for the aims of this work. For discussion of the equal characteristic case, which is expected to be analogous, see \cite[§4.8]{goresky2006codimensions}.

        \item If \( \thetoruslie_{w,\rootfunc} \neq \emptyset \), then for any \( m \in \tfrac{1}{a} \ints_{\geq 0} \), the set \( \explicitset{\alpha \in \theroots \mid \rootfunc(\alpha) \leq m} \subset \theroots \) is the set of roots of a Levi subalgebra of \( \thelie \).

        \item The group \( \theweylgroup \) acts on \( \arcspace(\thetoruslie') \) in a way that preserves the stratification. Explicitly, for any \( u \in \theweylgroup \), we have an isomorphism
        \[
        u \colon \thetoruslie_{w,\rootfunc} \xrightarrow{\sim} \thetoruslie_{uwu^{-1}, u(\rootfunc)}.
        \]

        \item \label{stratoftorus:twistedstrataresmooth} The nonempty strata \( \thetoruslie_{w,\rootfunc} \) are strongly pro-perfectly smooth, by an argument analogous to \cite[3.3.3.(c)]{bouthier2022perverse} using Teichm\"uller representatives instead of power series expressions.
    \end{mylist}
\end{para}

\begin{para}{\bf The stratification of \( \arcspace (\chevalley) \).}
    We have a disjoint union
    \[
        \arcspace(\thetoruslie') \times_{\arcspace(\chevalley')} \arcspace(\chevalley) \simeq \bigsqcup_{w\in W} \arcspace(\thetoruslie_w),
    \]
    and we denote by \( \arcspace(\thetoruslie')_{\arcspace (\chevalley),(\thediscriminant; n)} \) the preimage of \( \arcspace(\chevalley)_{(\thediscriminant; n)} \).
\end{para}

\cref{Proofofchevalleyisetale} is devoted to the proof of the following proposition: 

\begin{proposition}\label{stratoftorus:chevalleyisetale}
    The morphism
    \[
        \pi \colon \arcspace(\thetoruslie')_{\arcspace(\chevalley),(\thediscriminant;n)} \to \arcspace(\chevalley)_{(\thediscriminant ;n)}
    \]
    is a \( \theweylgroup \)-torsor.
\end{proposition}

\begin{para}\label{stratoftorus:propertiesofwrstrata}{\bf Components of strata of $\arcspace(\chevalley)$.}
    \begin{mylist}
        \item Since \( \thetoruslie_{w,\rootfunc} \) are disjoint closed subschemes of \( \arcspace(\thetoruslie_w)_{(\thediscriminant ; \sum_{\alpha\in R} \rootfunc(\alpha))} \), and there are finitely many such, they are the connected components. Since a finite étale morphism is both closed and open, it sends connected components to connected components.
        
        Thus, for each \( (w,\rootfunc) \), there exists a connected component \( \chevalley_{w,\rootfunc} \) such that \( \thetoruslie_{w,\rootfunc} \) is its étale covering.
        
        \item The pairs \( (w,\rootfunc) \) and \( (w',\rootfunc') \) lie in the same \( \theweylgroup \)-orbit if and only if \( \thetoruslie_{w,\rootfunc} \) and \( \thetoruslie_{w',\rootfunc'} \) cover the same \( \chevalley_{w,\rootfunc} \).

        \item The stratum \( \arcspace(\chevalley)_{(\thediscriminant ; n)} \) decomposes as a union of connected components \( \chevalley_{w,\rootfunc} \).

        \item Each \( \chevalley_{w,\rootfunc} \) admits a perfectly placid presentation, as they are connected components of certain base changes of the morphism $\arcspace(\chevalley)\to \arcspace_N(\chevalley)$ for some $N>>0$.
        \item\label{stratoftorus:stratofchevalleyisequidim} By \cref{dimtry:prop:uo_equidim_cover_by_smooth_imply_uo_equidim} and \cref{stratoftorus:twistedstrataresmooth},  \( \chevalley_{w,\rootfunc} \to * \) is universally open and equidimensional.

        \item\label{stratoftorus:stratofchevalleyisboundedconst} As in \cite[\S 4.1.7(d)]{bouthier2022perverse}, the collection \( \explicitset{\chevalley_{w,\rootfunc}}_{w,\rootfunc} \) forms a bounded constructible stratification of \( \arcspace(\chevalley)_{\bullet} \).
    \end{mylist}
\end{para}

\begin{remark}
    In the equal characteristic setting, \cref{stratoftorus:stratofchevalleyisequidim} is stronger; compare \cite[3.3.4(f)]{bouthier2022perverse}. We suspect that in our situation, the analogous statement—that \( \chevalley_{w,\rootfunc} \) is strongly pro-perfectly smooth—also holds. However, we did not verify this, as uo-equidimensionality suffices for our purposes.
\end{remark}


\begin{para}\label{codimofstrata:def:quantities}{\bf Some quantities.}
    \begin{mylist}
        \item\label{codimofstrata:def:quantities:pfplocallyofpurecodim}
        By the analogous statement to \cite[\S 3.2.2.(b)]{bouthier2022perverse}, both inclusions \( \thetoruslie_{w,\rootfunc} \subset \arcspace(\thetoruslie_w) \) and \( \chevalley_{w,\rootfunc} \subset \arcspace(\chevalley) \) are pfp-locally closed embeddings of connected perfect subschemes. By \cref{dimtry:prop:uo_equidim_cover_by_smooth_imply_uo_equidim}, the morphisms \( \thetoruslie_{w,\rootfunc} \to * \) and \( \chevalley_{w,\rootfunc} \to * \) are uo-equidimensional. Hence, by \cref{dimthry:prop:pfpfbetweenplacidpresentedcorollaries}, both inclusions are of pure codimension. We define
        \[
            a_{w,\rootfunc} := \codim_{\arcspace(\thetoruslie_w)} \thetoruslie_{w,\rootfunc}, \quad
            b_{w,\rootfunc} := \codim_{\arcspace(\chevalley)} \chevalley_{w,\rootfunc}.
        \]

        \item Let \( r := \dim \thetoruslie \) be the rank of \( G \). We define
        \[
            c_w := r - \dim \thetoruslie^w.
        \]

        \item Define
        \[
            d_\rootfunc := \sum_{\alpha \in \theroots} \rootfunc(\alpha).
        \]

        \item Set
        \[
            \delta_{w,\rootfunc} := \frac{d_\rootfunc - c_w}{2}.
        \]

        \item Define
        \[
            e_{w,\rootfunc} := \frac{d_\rootfunc + c_w}{2} = \delta_{w,\rootfunc} + c_w.
        \]
    \end{mylist}
\end{para}

\begin{proposition}\label{codimofstrata:prop:valofjacobian}
    Choose identifications $\thetoruslie_w \simeq \affinespace_\basicints^r$ and $\chevalley\simeq \affinespace_\basicints^r$, a pair $(w,\rootfunc)$ and an element \( x \in\thetoruslie_{w,\rootfunc}(\resfield)\subset\thetoruslie_w(\basicints) = \arcspace(\thetoruslie_w)(\resfield) \).  Then, denoting the Jacobian matrix by \( \differential_x \pi = \left( \frac{\partial \pi_i}{\partial x_i} \right) \), the valuation of \( \det(\differential_x \pi) \) with respect to \( \uniformizer \) is given by
    \[
        v_{\uniformizer}(\det(\differential_x \pi)) = e_{w,\rootfunc}.
    \]
    
\end{proposition}

\begin{proof}
    The proof is identical to that of \cite[Lemma 8.2.1]{goresky2006codimensions}.
\end{proof}

The proof of the following corollary is based on  \cref{appendixcodim:main}, shown in \cref{appendixcodim}. Note that  the original argument of \cite{goresky2006codimensions} does not apply here, as it uses tangent spaces, which are zero in our setting (see \cref{perfschemes:prop:notangent}).

\begin{corollary}\label{codimofstrata:def:eqofcodim}
    For every \( (w, \rootfunc) \), we have the equality
    \[
        b_{w,\rootfunc} = e_{w,\rootfunc} + a_{w,\rootfunc} = \delta_{w,\rootfunc} + a_{w,\rootfunc} + c_w.
    \]
\end{corollary}

\begin{proof}
    We apply \cref{appendixcodim:main} to the morphism \( \pi \colon \thetoruslie_w \to \chevalley \) and perfect subschemes \( \thetoruslie_{w,\rootfunc} \) and \( \chevalley_{w,\rootfunc} \). The assumptions of \cref{appendixcodim:main} are satisfied by \cref{codimofstrata:def:quantities:pfplocallyofpurecodim}, the identification \( \chevalley \simeq \affinespace^n_\basicints \), \cref{stratoftorus:chevalleyisetale} and \cref{codimofstrata:def:quantities:pfplocallyofpurecodim}. The valuation of the Jacobian determinant is given in \cref{codimofstrata:prop:valofjacobian}.
\end{proof}

\begin{corollary}\label{codimofstrata:def:smallnessinequality}
    For every \( (w, \rootfunc) \), we have the inequality
    \[
        b_{w,\rootfunc} \geq \delta_{w,\rootfunc},
    \]
    with equality if and only if \( w = 1 \) and \( \rootfunc = 0 \).
\end{corollary}

\begin{proof}
    The result follows from \cref{codimofstrata:def:eqofcodim}.
\end{proof}

\begin{para}{\bf The topologically nilpotent locus.}
    \begin{mylist}
        \item We define
        \[
            \arcspace(\chevalley)_{\topnilp} := \arcspace(\chevalley) \times_{\arcspace_0(\chevalley)} \{0\}.
        \]
        Then \( \arcspace(\chevalley)_{\topnilp} \subset \arcspace(\chevalley) \) is a pfp-closed embedding of a closed subscheme, which is strongly pro-perfectly smooth of pure codimension \( \dim \chevalley = r \).

        \item As before, \( \arcspace(\thetoruslie_w)_{\topnilp} \) is a connected affine closed perfect subscheme of \( \arcspace(\thetoruslie_w) \) of pure codimension \( \dim \thetoruslie^w = r - c_w \).

        \item Note that either \( \chevalley_{w,\rootfunc} \subset \arcspace(\chevalley)_{\topnilp} \) or \( \chevalley_{w,\rootfunc} \subset \arcspace(\chevalley) \smallsetminus \arcspace(\chevalley)_{\topnilp} \), and the same holds for \( \thetoruslie_{w,\rootfunc} \). In the first case, we say that \( (w,\rootfunc) > 0 \).

        \item For \( (w,\rootfunc) > 0 \), define
        \[
            b_{w,\rootfunc}^+ := b_{w,\rootfunc} - r, \quad
            a_{w,\rootfunc}^+ := a_{w,\rootfunc} - (r - c_w).
        \]

        \item \label{codimofstrata:def:topnilpquantities}
        We conclude that
        \[
            b_{w,\rootfunc}^+ = \codim_{\arcspace(\chevalley)_{\topnilp}} (\chevalley_{w,\rootfunc}), \quad
            a_{w,\rootfunc}^+ = \codim_{\arcspace(\thetoruslie_w)_{\topnilp}} (\thetoruslie_{w,\rootfunc}).
        \]
    \end{mylist}
\end{para}

\begin{corollary}\label{codimofstrata:def:topnilpinequality}
    For \( (w,\rootfunc) > 0 \), we have
    \[
        b_{w,\rootfunc}^+ = a_{w,\rootfunc}^+ + \delta_{w,\rootfunc}.
    \]
    In particular, we have \( b_{w,\rootfunc}^+ \geq \delta_{w,\rootfunc} \), and the equality holds if and only if \( \thetoruslie_{w,\rootfunc} \subset \arcspace(\thetoruslie_w)_{\topnilp} \) is an open stratum.
\end{corollary}

\begin{proof}
    The result follows from \cref{codimofstrata:def:eqofcodim}.
\end{proof}

\begin{para}{\bf Stratifications on \( \arcspace(\thelie) \) and on \( \liewahory \).}
    \begin{mylist}
        \item We let
        \[
            \thechevalleymap_n := \arcspace_n(\thechevalleymap) \colon \arcspace_n(\thelie) \to \arcspace_n(\chevalley),
        \]
        and denote its restriction to \( \liewahory_n \) by
        \[
            \theiwahorychevalleymap_n \colon \liewahory_n \to \arcspace_n(\chevalley).
        \]

        \item For each \( (w, \rootfunc) \), we let
        \[
            \thelie_{w,\rootfunc} \subset \arcspace(\thelie), \quad
            \liewahory_{w,\rootfunc} \subset \liewahory
        \]
        denote the preimages of \( \chevalley_{w,\rootfunc} \).

        \item\label{codimofstrata:prop:stratofliewahoryandlie}
        The subschemes \( \thelie_{w,\rootfunc} \subset \arcspace(\thelie) \) and \( \liewahory_{w,\rootfunc} \subset \liewahory \) form a bounded constructible stratification of \( \thelie_\bullet \) and \( \liewahory_\bullet \), respectively (\cref{stratification:prop:pullbackofstrat} and \cref{stratoftorus:stratofchevalleyisboundedconst}).
    \end{mylist}
\end{para}

The proof of the following theorem is given in \cref{flatnessofchevalley}.

\begin{theorem}\label{codimofstrata:prop:finitechevalleyflat}
    The maps
    \[
        \thechevalleymap_n \colon \arcspace_n(\thelie) \to \arcspace_n(\chevalley)
        \quad \text{and} \quad
        \theiwahorychevalleymap_n \colon \liewahory_n \to \arcspace_n(\chevalley)
    \]
    are flat.
\end{theorem}

The following corollary is of independent interest and is not used to prove the main theorem.

\begin{corollary}\label{codimofstrata:prop:chevalleyflat}
    The maps
    \[
        \thechevalleymap \colon \arcspace(\thelie) \to \arcspace(\chevalley)
        \quad \text{and} \quad
        \theiwahorychevalleymap \colon \liewahory \to \arcspace(\chevalley)
    \]
    are flat and uo-equidimensional.
\end{corollary}

\begin{proof}
    Flatness is preserved under filtered limits (since \( \operatorname{Tor}(-,-) \) commutes with filtered colimits). Therefore, the $\chi$ and $\nu$ are flat by \cref{codimofstrata:prop:finitechevalleyflat}.

    Next, since \( \thechevalleymap_n \) and \( \theiwahorychevalleymap_n \) are flat and perfectly finitely presented, they are uo-equidimensional by \cref{dimthry:prop:flatisuoequidim}. Therefore $\chi$ and $\nu$ are uo-equidimensional by \cite[A.4.6(b)]{bouthier2022perverse}. 
\end{proof}

\begin{corollary}\label{wittaffinrspringer:prop:chevalleyareflatanduoequidimonstrata}
    \begin{mylist}
        \item The perfectly finitely presented locally closed subschemes
        \[
            \liewahory_{w,\rootfunc} \subset \liewahory \quad \text{and} \quad \thelie_{w,\rootfunc} \subset \arcspace(\thelie)
        \]
        are of pure codimension \( b_{w,\rootfunc} \).
        
        \item The induced morphisms
        \[
            \thechevalleymap_{w,\rootfunc} \colon \thelie_{w,\rootfunc} \to \chevalley_{w,\rootfunc}
            \quad \text{and} \quad
            \theiwahorychevalleymap_{w,\rootfunc} \colon \liewahory_{w,\rootfunc} \to \chevalley_{w,\rootfunc}
        \]
        are flat and uo-equidimensional.
    \end{mylist}
\end{corollary}

\begin{proof}
    The result follows from \cref{codimofstrata:prop:finitechevalleyflat} by the same argument as in \cite[Corollary 3.4.9]{bouthier2022perverse}. The results of \cite[\S 2.3.6 and 2.3.7]{bouthier2022perverse} on base change along pro uo-equidimensional morphisms apply in our setting without modification.
\end{proof}

\section{Proof of \texorpdfstring{\cref{stratoftorus:chevalleyisetale}}{Proposition 7.12}} \label{Proofofchevalleyisetale}

To prove the proposition, we repeat the argument of \cite[Theorem 7.2.5]{bouthier2022perversesheavesinfinitedimensionalstacksv4}, which in its turn is motivated by \cite[\S 11.1]{goresky2006codimensions}.


\begin{para}
    Let $X =\arcspace(\thetoruslie')_{\arcspace(\chevalley),(\thediscriminant;n)}$ and let $Y = \arcspace(\chevalley)_{(\thediscriminant ; n)}$. Recall that $X,Y$ are perfect affine schemes. We want to show the morphism $X\to Y$ is a $W$-torsor. 
\end{para}


\begin{proposition}\label{proofofetale:torsoronpoints}
     For every algebraically closed field $F/\resfield$, the  morphism $\pi: X(F)\to Y(F)$ is surjective, and every fiber is a $W$-torsor.
\end{proposition}

\begin{proof}

Recall that the points of  $Y(F)$ are elements $y$ of $\chevalley(\witt(F))\cap \chevalley^{\regsemisimple}(\witt(F)[1/p])$ such that
$v(\thediscriminant(y))=n$, while the points of  $X(F)$ are elements $X$ of $\thetoruslie(\witt(F)[\uniformizer^{1/a}])\cap \thetoruslie^{rs}(\witt(F)[\uniformizer^{1/a}][1/p])$
such that $v(\thediscriminant(x))=n$ and $\sigma(x)=w^{-1}(x)$ for some $w\in W$.

Since the $\pi:\thetoruslie^{rs}\to \chevalley^{rs}$ is a $W$-torsor, every fiber of  $X(F)\to Y(F)$ is either a $W$-torsor or empty. Thus it suffices to show that $\pi: X(F)\to Y(F)$ is surjective. Fix
$y\in Y(F)$.

Since $y\in \chevalley^{\regsemisimple}(\witt(F)[1/p])$ and $\pi:\thetoruslie^{rs}\to \chevalley^{rs}$ is a $W$-torsor, there exists a finite Galois extension $M/\witt(F)[1/p]$ of degree $m\mid a$ (see \cref{wittloopandarc:def:constants}) such that $y\in \pi(\thetoruslie^{\regsemisimple}(M))$. Since $a$ is invertible in $F$, we have $M\simeq \witt(F)[1/p,\uniformizer^{1/m}]$, thus there exists $x\in \thetoruslie^{\regsemisimple}(\witt(F)[1/p,\uniformizer^{1/m}])$ such that $\pi(x)=y$.
Moreover, $\pi(\sigma(x))=\sigma(\pi(x))=\sigma(y)=y$. Thus there exists  $w\in W$ such that $\sigma(x)=w^{-1}(x)$.
Finally, since $y\in \chevalley(\witt(F))$ and $\thetoruslie\to \chevalley$ is finite, hence proper, it follows from the valuative criterion that $x\in \thetoruslie(\witt(F)[\uniformizer^{1/m}])$.
\end{proof}

\begin{proposition}
    It suffices to show that the  morphism  $\pi:X\to Y$ is \'etale, and so, in particular perfectly finitely presented.
\end{proposition}

\begin{proof}
Assume that $\pi$ is \'etale. Since $\pi$ is surjective (by Step 1), it is faithfully flat. Thus
it suffices to show that the  morphism $d: W\times X\to X\times_Y X:(w,x)\mapsto (wx,x)$ is an isomorphism.
Since $d$ is surjective by Step 1, it suffices to show that $d$ is an open embedding.

Since $\pi$ is \'etale, the diagonal  morphism $X\to X\times_Y X$ is an open embedding. Therefore the  morphism
$d_w: X\to X\times_Y X:x\mapsto (wx,x)$ is an open embedding for all $W$. Moreover, by \cref{proofofetale:torsoronpoints}, the images of the $d_w$'s do not intersect. Thus $d$ is an open embedding, and we are done.
\end{proof}


\begin{proposition}
    It suffices to show that the  morphism $\thetoruslie_{\rootfunc}\to \arcspace(\chevalley)_{(\thediscriminant;n)}$, induced by $\pi$, is \'etale for all $\rootfunc:R\to \nats$. 
\end{proposition}

\begin{proof}

Since $X$ is a disjoint union of  $\arcspace(\thetoruslie_w)_{(\thediscriminant ;n)}$ for $w\in \theweylgroup$, it suffices to show that each  morphism $\arcspace(\thetoruslie_w)_{(\thediscriminant ;n)}\to \arcspace(\chevalley)_{(\thediscriminant ; n)}$ is \'etale.

Assume that  morphism $\thetoruslie_{\rootfunc}\to \arcspace(\chevalley)_{(\thediscriminant ;n)}$ is \'etale for all $\rootfunc$.
Since  $\arcspace(\thetoruslie)_{(\thediscriminant ; n)}$ is a disjoint union of the  $\thetoruslie_\rootfunc$'s, we conclude that the  morphism
$\arcspace(\thetoruslie)_{(\thediscriminant ; n)}\to \arcspace(\chevalley)_{(\thediscriminant ; n)}$ is \'etale.
Applying this to $\thetoruslie'$ instead of $\thetoruslie$, we conclude that  the morphism
$\arcspace(\thetoruslie')_{(\thediscriminant ;n)}\to \arcspace(\chevalley')_{(\thediscriminant ;n)}$ is \'etale.

Finally, since $\arcspace(\thetoruslie_w)_{(\thediscriminant;n)}$ (resp. $\arcspace(\chevalley)_{(\thediscriminant;n)}$) is the scheme of fixed points $w\sigma$ (resp. $\sigma$  inside $\arcspace(\thetoruslie')_{(\thediscriminant;n)}$ (resp. $\arcspace(\chevalley')_{(\thediscriminant;n)}$), the assertion follows from \cref{proofofetale:lemmafixedpoints} below.
\end{proof}

\begin{lemma} \label{proofofetale:lemmafixedpoints}
Let $f:T\to S$ be a separated \'etale morphism of perfect schemes, and let
$\phi_T\in\End T$ and $\phi_S\in\End S$ be endomorphisms such that $f\circ\phi_T=\phi_S\circ f$.
Then the induced morphism between schemes of fixed points $f^{\phi}:T^{\phi_T}\to S^{\phi_S}$ is \'etale.
\end{lemma}

\begin{proof}
(compare \cite[15.4.2(3)]{goresky2006codimensions}). Restricting $f$ to the $S^{\phi_S}\subset S$, we may assume that $\phi_S$ is the identity.
Set $\phi:=\phi_T$. Then we claim that the embedding $\iota_{\phi}:T^{\phi}\to T$ is clopen (that is, open and closed). Hence $f^{\iota}=f|_{T^{\phi}}$ is \'etale, as claimed.

The diagonal morphism $\Delta_f:T\to T\times_S T$ is an open embedding, because $f$ is \'etale, hence a clopen embedding, since $f$ is separated. Taking pullback with respect to $(\Id,\phi):T\to T\times_S T$, we conclude that the morphism $T^{\phi}\to T$ is a clopen embedding as well.
\end{proof}

\begin{remark}
Since $f$ is formally \'etale, it is immediate to show that $f^{\phi}$ is formally \'etale as well. So the main point of \cref{proofofetale:lemmafixedpoints} was to show that $f^{\phi}$ is finitely presented.
\end{remark}

\noindent For the rest of the proof, we follow \cite[section 11.1]{goresky2006codimensions} very closely: 
\vskip 4truept

\begin{proposition} \label{P:redss}
    We may assume that $\thelie$ is semisimple.
\end{proposition}

\begin{proof}
Indeed, we have a decomposition $\thelie=\thelie_{ss}\times \theliecenter$, where $\thelie_{ss}$ is the derived algebra of $\thelie$ and $\theliecenter$ is the center of $\thelie$. Moreover, the morphism  $\arcspace(\thetoruslie)_{(\thediscriminant;n)}\to \arcspace(\chevalley)_{(\thediscriminant;n)}$ decomposes as a product of the corresponding morphism for $\thelie_{ss}$ and the identity morphism on $\arcspace(\theliecenter)$. Thus the assertion for $\thelie$ follows from that for $\thelie_{ss}$.
\end{proof}

\begin{proposition}
     It suffices to assume that $\min \rootfunc=0$.
\end{proposition}

\begin{proof}
By \cref{P:redss}, we can assume that $\thelie$ is semisimple. We set $s:=\min \rootfunc$, and $\rootfunc':=\rootfunc-s$.
Recall that the morphism $\thetoruslie\to \chevalley$ is $\multiplicativegroup$-equivariant. Thus the morphism $\loopspace(\thetoruslie)\to \loopspace(\chevalley)$ is $\loopspace(\multiplicativegroup)$-equivariant. Moreover, the element $\uniformizer^s\in \loopspace(\multiplicativegroup)$ induces isomorphisms
$\thetoruslie_{\rootfunc'}\simeq \thetoruslie_\rootfunc$ and $\arcspace(\chevalley)_{(\thediscriminant; n-s|\theroots|)}\simeq \arcspace(\chevalley)_{(\thediscriminant ; n)}$. Therefore the assertion for $\rootfunc'$ implies that for $\rootfunc$.
\end{proof}


\begin{proposition}
    If $\thelie$ is semi-simple and $\min \rootfunc = 0$, then $\thetoruslie_{\rootfunc}\to \arcspace(\chevalley)_{(\thediscriminant;n)}$, induced by $\pi$, is \'etale.
\end{proposition}

\begin{proof}
    Since $\min\rootfunc=0$, the set $R':=\{\alpha\in R\,|\,\rootfunc(\alpha)>0\}$ is a root system of a proper Levi subgroup $M$ of $G$ (by \cref{stratoftorus:twistedstrat}). Consider the Chevalley space $\chevalley_M$ of $M$. Then the discriminant function $\thediscriminant\in k[\chevalley_M]$ decomposes as $\thediscriminant=\thediscriminant_M \thediscriminant^M$, where $\thediscriminant_M=\prod_{\alpha\in R'}d\alpha$ and $\thediscriminant^M=\prod_{\alpha\in R\smallsetminus R'}d\alpha$.
    Then the morphism  $\pi:\thetoruslie_\rootfunc\to \arcspace(\chevalley)_{(\thediscriminant ; n)}$ decomposes as
    \[
    \thetoruslie_\rootfunc\to \arcspace(\chevalley_M)_{(\thediscriminant_M ; n),(\thediscriminant^M ; 0)}=\arcspace(\chevalley_M)_{(\thediscriminant ; n),(\thediscriminant^M ; 0)}\to \arcspace(\chevalley)_{(\thediscriminant ; n)}.
    \]
    So it remains to show that both morphisms are \'etale. The assertion for the first morphism follows by induction on $|\theroots|$ (as we assumed that $\liealg[m]$ has nontrivial center), so it remains to show the assertion for the second morphism.
        
    \vskip 4truept
    
    Consider the open subscheme $\chevalley^{\operatorname{reg}/\thelie}_M:=(\chevalley_M)_{\thediscriminant^M}\subset\chevalley_M$. Since \[\arcspace( \chevalley)^{\operatorname{reg}/\thelie}_M = (\arcspace (\chevalley)_M)_{(\thediscriminant^M;0)}\subset \arcspace (\chevalley)\] is Zariski open, it suffices to show that the morphism
    $\arcspace(\chevalley^{\operatorname{reg}/\thelie}_M)_{(\thediscriminant ; n)}\to \arcspace(\chevalley)_{(\thediscriminant ; n)}$ is \'etale.  We claim that the entire morphism $\arcspace(\chevalley^{\operatorname{reg}/\thelie}_M)\to \arcspace(\chevalley)$ is \'etale. Namely, the morphism  $\chevalley^{\operatorname{reg}/\thelie}_M\to\chevalley$ is \'etale, so the assertion follows from \cref{wittloopandarc:properties:arpreserveetale}.
\end{proof}

\section{The Witt affine Grothendieck--Springer fibration}\label{wittaffinrspringer}

\begin{para}{\bf The Witt affine Grothendieck–Springer fibration.}\label{wittaffinrspringer:gneral:definitionoffib}
    \begin{mylist}
        \item We define
        \[
            \compactelements := \loopspace( \thechevalleymap)^{-1}(\arcspace(\chevalley))\subset \loopspace(\thelie).
        \]
        We call this the space of \emph{compact elements} in \( \loopspace(\thelie) \).

        \item \label{wittaffinrspringer:gneral:compisperfectlyplacid} \( \compactelements \) is an ind-perfectly placid perfect ind-scheme by the same argument as \cite[4.1.1(a)]{bouthier2022perverse}. 

        \item Let \( \springertotalspace := \loopspace(\thegroup) \times^{\theiwahori} \liewahory \), i.e., the quotient of \( \thegroup \times \liewahory \) by \( \iwahory \) under the action
        \[
            h \cdot (g, x) := (g h^{-1}, \theadjoint_h(x)).
        \]

        \item We have an isomorphism of perfect \( \infty \)-stacks
        \(
            [\springertotalspace / \loopspace(\thegroup)] \simeq [\liewahory / (I, \theadjoint)]
        \) (the argument of \cite[4.1.1(c)]{bouthier2022perverse} works here).

        \item There is a natural morphism
        \[
            \thespringermap \colon \springertotalspace \to \compactelements,
        \]
        called the \emph{Witt affine Grothendieck–Springer fibration}, defined by \( [g, x] \mapsto \theadjoint_g(x) \). This morphism is \( \loopspace(\thegroup) \)-equivariant with respect to the left action on the source and the adjoint action on the target. Hence, it induces a morphism of perfect $\infty$-stacks
        \[
            \overline{\thespringermap} \colon [\springertotalspace / \loopspace(\thegroup)] \to [\compactelements / \loopspace(\thegroup)].
        \]

        \item The fibers of \( \thespringermap \) are called the \emph{Witt vector affine Springer fibers}, and are studied in \cite{chi2024witt}, where their dimensions are computed. Specifically, over the stratum \( \chevalley_{w,\rootfunc} \), the fibers of \( \thespringermap \) have dimension \( \delta_{w,\rootfunc} \) (see \cref{codimofstrata:def:quantities}).
    \end{mylist}
\end{para}

\begin{para}{\bf The stratification on the fibration.}
    \begin{mylist}
        \item Define
        \[
            \compactelements_{w,\rootfunc} := \loopspace (\thechevalleymap)^{-1}(\chevalley_{w,\rootfunc}), \quad
            \springertotalspace_{w,\rootfunc} := \thespringermap^{-1}(\compactelements_{w,\rootfunc}),
        \]
        and let \( \thespringermap_{w,\rootfunc} \) denote the restriction of \( \thespringermap \) to these strata.

        \item By definition, \( \compactelements_{w,\rootfunc} \) and \( \springertotalspace_{w,\rootfunc} \) are \( \loopspace(\thegroup) \)-stable, so the fibration restricts to a morphism of perfect $\infty$-stacks
        \[
            \overline{\thespringermap}_{w,\rootfunc} \colon
            [\springertotalspace_{w,\rootfunc} / \loopspace(\thegroup)] \to
            [\compactelements_{w,\rootfunc} / \loopspace(\thegroup)].
        \]

        \item Let
        \[
            \thespringermap_{\topnilp} \colon \springertotalspace_{\topnilp} \to \compactelements_{\topnilp}, \quad
            \thespringermap_{\bullet} \colon \springertotalspace_{\bullet} \to \compactelements_{\bullet},
        \]
        and
        \[
            \thespringermap_{\topnilp,\bullet} \colon \springertotalspace_{\topnilp,\bullet} \to \compactelements_{\topnilp,\bullet}
        \]
        denote the respective restrictions of \( \thespringermap \) over \( \chevalley_{\topnilp} \), \( \chevalley_{\bullet} \), and \( \chevalley_{\topnilp,\bullet} := \chevalley_{\topnilp}\cap\chevalley_{\bullet}\).

        \item Since \( \chevalley_{\topnilp} \), \( \chevalley_{\bullet} \), and \( \chevalley_{\topnilp,\bullet} \) are all \( \explicitset{\chevalley_{w,\rootfunc}}_{w,\rootfunc} \)-adapted locally closed perfect subschemes, the collection \( \explicitset{\chevalley_{w,\rootfunc}}_{w,\rootfunc} \) induces constructible stratifications on \( \compactelements_{\topnilp} \), \( \compactelements_{\bullet} \), and \( \compactelements_{\topnilp,\bullet} \).

        \item The stratum \( \chevalley_{\leq 0} \subset \arcspace(\chevalley) \) is a dense pfp-open  embedding, corresponding to the case \( w = 1 \), \( \rootfunc=0 \). Let
        \[
            \thespringermap_{\leq 0} \colon \springertotalspace_{\leq 0} \to \compactelements_{\leq 0}
        \]
        denote the restriction of the fibration over this open stratum.
    \end{mylist}
\end{para}

The following is the main theorem of this section.

\begin{theorem}\label{semismallness:prop:mainthm}
    \begin{mylist}
        \item The morphism
        \[
            \overline{\thespringermap} \colon [\springertotalspace / \loopspace(\thegroup)] \to [\compactelements / \loopspace(\thegroup)]
        \]
        is ind-pfp proper.

        \item The strata
        \[
            \{[\compactelements_{w,\rootfunc} / \loopspace(\thegroup)]\}_{w,\rootfunc}
        \]
        form placid constructible stratifications of both
        \[
            [\compactelements_{\bullet} / \loopspace(\thegroup)] \quad \text{and} \quad [\compactelements_{\topnilp,\bullet} / \loopspace(\thegroup)].
        \]

        \item The perfect \( \infty \)-stacks
        \[
            [\springertotalspace_{\bullet} / \loopspace(\thegroup)] \quad \text{and} \quad [\springertotalspace_{\topnilp,\bullet} / \loopspace(\thegroup)]
        \]
        are perfectly smooth.

        \item The strata
        \[
            [\compactelements_{w,\rootfunc} / \loopspace(\thegroup)] \subset [\springertotalspace_{\bullet} / \loopspace(\thegroup)]\text{ and }
            [\compactelements_{w,\rootfunc} / \loopspace(\thegroup)] \subset [\compactelements_{\topnilp,\bullet} / \loopspace(\thegroup)]
        \]
        are of pure codimension (in the sense of \cref{dimthry:prop:purecodimstacks}) \( b_{w,\rootfunc} \) and \( b_{w,\rootfunc}^+ \), respectively (see \cref{codimofstrata:def:quantities} and \cref{codimofstrata:def:topnilpquantities}).

        \item For each \( (w,\rootfunc) \), the morphism
        \[
            \overline{\thespringermap}_{w,\rootfunc} \colon [\springertotalspace_{w,\rootfunc} / \loopspace(\thegroup)] \to [\compactelements_{w,\rootfunc} / \loopspace(\thegroup)]
        \]
        is uo-equidimensional, locally pfp-representable (see \cref{perfstacks:def:pfprepresentable}), and of relative dimension \( \delta_{w,\rootfunc} \) (see \cref{dimthry:prop:constantrelativedimensionbetweenplacid} and \cref{codimofstrata:def:quantities}).

        \item Over the stratum \( \compactelements_{\leq 0} \), corresponding to \( w = 1 \), \( \rootfunc=0 \), the morphism
        \[
            \overline{\thespringermap}_{\leq 0} \colon [\springertotalspace_{\leq 0} / \loopspace(\thegroup)] \to [\compactelements_{\leq 0} / \loopspace(\thegroup)]
        \]
        is a \( \theextendedweylgroup \)-torsor.
    \end{mylist}
\end{theorem}

\begin{proof}
    The proof of (a) is analogous to \cite[Lemma 4.1.4]{bouthier2022perverse}; one uses that the Witt vector affine Grassmannian is ind-pfp proper, as shown in \cite{zhu2017affine} (see \cite{bhatt2017projectivity} for a stronger result).

    Part (b) follows from \cref{codimofstrata:prop:stratofliewahoryandlie} and the isomorphism
    \[
        [\springertotalspace / \loopspace(\thegroup)] \simeq [\liewahory / (I, \theadjoint)],
    \]
    as in \cite[Lemma 4.4.2(a)]{bouthier2022perverse} and \cite[Lemma 4.4.4(a)]{bouthier2022perverse}.

    Part (c) is proved as in \cite[Lemma 4.4.2(c)]{bouthier2022perverse} and \cite[Lemma 4.4.4(c)]{bouthier2022perverse}.

    Part (d) follows from \cref{wittaffinrspringer:prop:chevalleyareflatanduoequidimonstrata}(a) in the same way that the corresponding statements follow from \cite[Corollary 3.4.2]{bouthier2022perverse} in the equal characteristic setting.

    Part (e) is proved as in \cite[Lemma 4.3.4(b)]{bouthier2022perverse}. The fact that \( \overline{\thespringermap}_{w,\rootfunc} \) is locally pfp-representable uses an analogue of \cite[Theorem 4.3.3]{bouthier2022perverse} in our setting, which is proved as in \cite[\S B.2]{bouthier2022perverse}. The only difference is that, in our setting, there is no notion of reduction; thus, for instance, our analogue of \cite[Theorem 4.1.9]{bouthier2022perverse} is \cref{wittaffinrspringer:gneral:looppreservesquot} below, which has a cleaner formulation and a slightly simpler proof.

    To show that \( \overline{\thespringermap}_{w,\rootfunc} \) is uo-equidimensional of relative dimension \( \delta_{w,\rootfunc} \), we use \cref{wittaffinrspringer:prop:chevalleyareflatanduoequidimonstrata}(b) in place of \cite[Corollary 3.4.9(b)]{bouthier2022perverse}, and we appeal to \cite[Theorem 1.1]{chi2024witt} instead of the classical dimension formula for affine Springer fibers.

    Part (f) is proved as in \cite[Corollary 4.2.5]{bouthier2022perverse}. Namely, the torsor structure arises from an action of \( \Lambda_T \), constructed using \cref{wittloopandarc:example:loopofGm}.
\end{proof}

The following theorem is the analogue of \cite[Theorem 4.1.9]{bouthier2022perverse} and plays a crucial role in the proof of   \cref{semismallness:prop:mainthm}(b). Working in the perfect setting cleans up the formulation and simplifies its proof (see \cite[\S B.2]{bouthier2022perverse}).

\begin{theorem} \label{wittaffinrspringer:gneral:looppreservesquot}
    For every (not necessarily split) maximal torus \( S \subset \thegroup \), the natural projection
    \[
        \psi_S \colon [\loopspace (\thegroup) / \loopspace (S)] \to \loopspace(\thegroup / S)
    \]
    is an isomorphism.
\end{theorem}

\begin{proof}
    We may assume that \( S \) is defined over \( \basicints \).

    Consider the composition
    \[
        \loopspace (\thegroup )/ \arcspace (S) \to [\loopspace (\thegroup) / \loopspace (S)] \to \loopspace(\thegroup / S).
    \]
    Since \( \loopspace (S) \simeq \arcspace (S) \times \Lambda_S \) (by \cref{wittloopandarc:example:loopofGm}), the first morphism is a \( \Lambda_S \)-torsor. We will show that the full composition is an ind-pfp-proper-locally schematic morphism, which is a \( \Lambda_S \)-torsor on \( F \)-points for any algebraically closed field \( F / \resfield \). The result then follows from \cref{perfstacks:prop:torsoronpoints}.

    The proofs that it is ind-pfp-proper-locally schematic and a torsor on \( F \)-points are identical to those in \cite[Claim B.2.2]{bouthier2022perverse}.

    Note that we really use \cite{bhatt2017projectivity}, rather than the weaker result in \cite{zhu2017affine}, to conclude that \( \loopspace (\thegroup) / \iwahory \) is ind-projective (and not merely ind-proper). This ensures that the composition is ind-pfp-proper-locally schematic.
\end{proof}

\begin{corollary}\label{semismallness:prop:maincor}
    The morphism
    \[
        \thespringermap_{\bullet} \colon [\springertotalspace_{\bullet} / \loopspace (\thegroup)] \to [\compactelements_{\bullet} / \loopspace (\thegroup)]
    \]
    is \( [\compactelements_{\leq 0} / \loopspace (\thegroup)] \)-small, and the morphism
    \[
        \thespringermap_{\topnilp,\bullet} \colon [\springertotalspace_{\topnilp,\bullet} / \loopspace (\thegroup)] \to [\compactelements_{\topnilp,\bullet} / \loopspace (\thegroup)]
    \]
    is semi-small (see \cref{stratification:def:semismall}).
\end{corollary}

\begin{proof}
    This follows from \cref{semismallness:prop:mainthm}, \cref{codimofstrata:def:smallnessinequality}, and \cref{codimofstrata:def:topnilpinequality}.
\end{proof}

\section{Perverse \texorpdfstring{$t$}{t}-structures on infinite-dimensional perfect \texorpdfstring{\( \infty \)}{infinity}-stacks}\label{perversetstructures}

\begin{para}{\bf The $\infty$-category \( \prl \).}
    Let \( \prl \) denote the $\infty$-category of stable presentable \( \qlbar \)-linear $\infty$-categories and continuous functors, where “continuous” means preserving small colimits.
\end{para}

\begin{para}{\bf Ind-constructible \( ! \)-sheaves on non-perfect \( \infty \)-stacks.}
    We recall the definition of \( \indconstructible \colon \stacks^{\op} \to \prl \), the functor from the $\infty$-category of (not necessarily perfect) \( \infty \)-stacks to the $\infty$-category of stable presentable \( \overline{\mathbb{Q}}_\ell \)-linear $\infty$-categories, constructed as in \cite{bouthier2022perverse}.

    \begin{mylist}
        \item We first define \( \indconstructible \colon (\algsp^{\fp})^{\op} \to \prl \) by sending an algebraic space \( X \), finitely presented over \( \spec \resfield \), to the ind-category of constructible sheaves on \( X \), and a morphism \( f \) to the pullback functor \( f^! \).

        \item We then take the left Kan extension of this functor along the inclusion \( (\algsp^{\fp})^{\op} \hookrightarrow (\algsp^{\qcqs})^{\op} \). Thus, for a qcqs algebraic space \( X \), presented as a filtered colimit \( X \simeq \lim_\alpha X_\alpha \) of algebraic spaces \( X_\alpha \) finitely presented over \( \spec \resfield \), we have \( \indconstructible(X) = \colim_\alpha \indconstructible(X_\alpha) \), and the colimit is independent of the presentation (see \cite[\S 5.2.4(b)]{bouthier2022perverse}).

        \item Finally, we take the right Kan extension of this functor along the {restricted} Yoneda embedding \( (\algsp^{\qcqs})^{\op} \hookrightarrow \stacks^{\op} \). Hence, for an \( \infty \)-stack \( \varstack{X} \) represented by a hypercovering \( \varstack{X} \simeq \colim_{\Delta^{\op}} X_{[m]} \), we have \( \indconstructible(\varstack{X}) \simeq \lim_{[m]\in\Delta} \indconstructible(X_{[m]}) \), and the limit is independent of the presentation (see \cite[\S 5.3.1(e)]{bouthier2022perverse}).
    \end{mylist}
\end{para}

\begin{para}\label{contructibleonperfect:prop:Dfactorsthroughperfection} {\bf The functor \( \indconstructible \) on perfect \( \infty \)-stacks.}
    \begin{mylist}

        \item By topological invariance of the étale site (see \cref{perfschemes:prop:topologicalinvarianceofetale}), the functor \( \indconstructible \) factors as 
        \[
        \begin{tikzcd}
            \stacks^{\op} \arrow[r] \arrow[d, "(-)_{\perf}"'] & \prl \\
            \perstacks^{\op} \arrow[ru, dashed]
        \end{tikzcd}
        \]
        that is, the canonical morphism \( \varstack{X}_{\perf} \to \varstack{X} \) is sent to an equivalence by \( \indconstructible \).

        \item We will denote the resulting functor $\perstacks^{\op}\to\prl$ also by \( \indconstructible \).
    \end{mylist}
\end{para}

\begin{definition}
    For any \( \infty \)-stack \( \varstack{X} \), we define its dualizing object by
    \[
        \omega_{\varstack{X}} := p^!(\qlbar),
    \]
    where \( p \colon \varstack{X} \to * \) is the unique morphism.
\end{definition}

The following will be the only input that differs from the equal characteristic case. This result was mentioned in \cite[\S A.3]{zhu2017affine} without proof. For convenience of the reader we add the proof.

\begin{proposition}\label{contructibleonperfect:prop:perfsmoothperfproperbasechange}{\bf Perfectly smooth and perfectly proper base change.}
    Let \( f \colon Z \to W \) and \( g \colon Y \to W \) be morphisms in \( \peralgsp^{\pfp} \). Then \( f^! \) and \( g^! \) admit left adjoints \( f_! \) and \( g_! \), respectively. Moreover, if we write the Cartesian square
    \[
    \begin{tikzcd}
        X \arrow[r, "f'"] \arrow[d, "g'"'] \arrow[dr, phantom, "\lrcorner", very near start] & Y \arrow[d, "g"] \\
        Z \arrow[r, "f"] & W
    \end{tikzcd}
    \]
    there is a natural base change morphism
    \[
        \beta_{f^!} \colon g'_! f'^! \to f^! g_!.
    \]
    If either \( g \) is perfectly proper or \( f \) is perfectly smooth, then \( \beta_{f^!} \) is an isomorphism.
\end{proposition}

\begin{proof}
    The first assertion, regarding the existence of left adjoints \( f_! \) and \( g_! \), follows immediately from taking pfp models.

    Suppose first that \( g \) is perfectly proper. Then, by \cref{perfschemes:prop:existsfinitetypemodel}, the square is the perfection of a Cartesian square of schemes finitely presented over \( \spec \resfield \). Moreover, the model of \( g \) is universally closed, since this property is reflected by perfection. The result then follows from \cref{contructibleonperfect:prop:Dfactorsthroughperfection} and the classical proper base change theorem.

    Now suppose that \( f \) is perfectly smooth. Since \( p^! \) is conservative for an étale surjective morphism \( p \colon Z' \to Z \), we may replace \( Z \) by a finite étale cover. Then, by the definition of a perfectly smooth morphism, it suffices to prove the assertion in the cases where \( f \) is either étale or the projection \( \affinespace^n_{\perf} \times Y \to Y \). In both cases, the assertion follows from the usual smooth base change theorem together with \cref{contructibleonperfect:prop:Dfactorsthroughperfection}.
\end{proof}

\begin{para}{\bf Gluing of sheaves.} 
    \begin{mylist}
        \item Let $\eta:\varstack{X}\into \varstack{Y}$ be a locally closed embedding of perfectly placid perfect $\infty$-stacks. We define $\eta_*$ as in \cite[Theorem 5.4.4]{bouthier2022perverse}.
        \item We say that a perfect $\infty$-stack $\varstack{Y}$ {\em admits gluing of sheaves}, if for every pfp-locally closed embedding $\eta:\varstack{X}\to \varstack{Y}$ the pushforward $\eta_*:\indconstructible(X)\to \indconstructible(Y)$ admits a left adjoint $\eta^*:\indconstructible(Y)\to \indconstructible(X)$.
        \item As in \cite[Lemma 5.5.4]{bouthier2022perverse}, every perfectly placid perfect $\infty$-stack $\varstack{X}$ admits a gluing of sheaves.
        \item\label{contructibleonperfect:prop:quotient_of_ind_placd_space_by_ind_placid_group_admits_gluing} As in \cite[Lemma 5.5.5]{bouthier2022perverse}, if $G$ is an ind-perfectly placid perfect group space acting on an ind perfectly placid ind algebraic space $X$, then $\varstack{X}:= [X/G]$ admits gluing of sheaves.
    \end{mylist}
\end{para}

\begin{para}{\bf The \( ! \)-adapted $t$-structure on perfectly finitely presented perfect algebraic spaces.}
    \begin{mylist}
        \item Let \( X \in \peralgsp^{\pfp} \). We equip the $\infty$-category \( \constructible(X) \), the bounded derived $\infty$-category of \( \ell \)-adic constructible sheaves on \( X \), with a canonical perverse $t$-structure
        \[
            ({}^p\constructible^{\leq 0}(X),{}^p\constructible^{\geq 0}(X))
        \]
        defined by the following properties:

        \begin{mylist}
            \item If \( Y \) is an equidimensional perfect algebraic space of dimension \( d \), perfectly finitely presented over \( \spec\resfield \), then the perverse $t$-structure on \( \constructible(Y) \) is the classical perverse $t$-structure shifted \( d \) times to the left. That is,
            \[
                K \in {}^p\constructible^{\leq 0}(Y) \iff K[-d] \in \constructible^{\leq 0}(Y)
            \]
            with respect to the classical perverse $t$-structure and similarly for ${}^p\constructible^{\geq 0}(Y)$. 

            \item For a general \( X \in \peralgsp^{\pfp} \), there is a canonical stratification $\eta_i:X_i \into X$ such that each $X_i$ is an equidimensional perfect algebraic space of dimension $i$ (see \cite[\S 2.1.1.(c)]{bouthier2022perverse} for the case of non perfect schemes, which is completely analogous to our case). Define
            \[
                K \in {}^p\constructible^{\leq 0}(X) \iff \forall i \;\;  \eta_i^*K\in{}^p\indconstructible^{\leq 0}(X_i)\text{ and } K \in {}^p\constructible^{\geq 0}(X) \iff \forall i \;\;  \eta_i^!K\in{}^p\indconstructible^{\geq 0}(X_i).
            \]
            See \cite[\S6.2.3]{bouthier2022perverse} for the proof that this construction defines a $t$-structure.
        \end{mylist}

        \item For any \( X \in \peralgsp^{\pfp} \), we have \( \indconstructible(X) = \Ind(\constructible(X)) \). Thus, we obtain a canonical $t$-structure
        \(({}^p\indconstructible^{\leq 0}(X),{}^p\indconstructible^{\geq 0}(X))\) on \( \indconstructible(X)\), 
        where \( K \in {}^p\indconstructible^{\leq 0}(X) \)  if it can be presented as a filtered colimit of objects in \( {}^p\constructible^{\leq 0}(X) \) 
        and similarly for \({}^p\indconstructible^{\geq 0}(X)\).
    \end{mylist}
\end{para}

\begin{para}\label{perversetstructures:def:placidpresentation}{\bf Perverse $t$-structures on algebraic spaces admitting a perfectly placid presentation.}
    \begin{mylist}
        \item Let \( f \colon X \to Y \) be an étale cover in \( \peralgsp^{\pfp} \). Then \( f^! \) is \( t \)-exact and conservative. Thus, \( K \in {}^p\indconstructible^{\leq 0}(Y) \) (resp. \( K \in {}^p\indconstructible^{\geq 0}(Y) \)) if and only if \( f^!K \in {}^p\indconstructible^{\leq 0}(X) \) (resp. \( f^!K \in {}^p\indconstructible^{\geq 0}(X) \)).

        {\color{black}\item\label{perversetstructures:def:texactnessforperfsmooth} Let \( f \colon X \to Y \) be a perfectly smooth morphism in \( \peralgsp^{\pfp} \). Then \( f^! \) is \( t \)-exact.}

        \item Given a perfectly placid presentation \( X \simeq \lim X_\alpha \), there exists a unique \( t \)-structure on \( \indconstructible(X) \) such that each pullback \( p_\alpha^! \colon \indconstructible(X) \to \indconstructible(X_\alpha) \) is \( t \)-exact.
    \end{mylist}
\end{para}

\begin{proof}
    (a) follows from the classical case: since étale morphisms preserve dimensions, the shifts in the perverse $t$-structure are unaffected.

    We now prove (b). Since \( p^! \) is conservative for an étale surjective morphism \( p \colon Z' \to Z \), we may replace \( Z \) by a finite étale cover. Then, by the definition of perfectly smooth morphisms, it suffices to consider the cases where \( f \) is either étale or the projection \( \affinespace^n_{\perf} \times Y \to Y \). In both cases, \( t \)-exactness follows from the non perfect case and from \cref{contructibleonperfect:prop:Dfactorsthroughperfection}.

    (c) is a particular example of  a general fact: a filtered colimit in \( \prl \) with \( t \)-exact transition functors carries an induced \( t \)-structure (see \cite[Lemma 6.1.3]{bouthier2022perverse}).
\end{proof}

The proof of the following proposition differs from it's counterpart in the non-perfect setting (compare \cite[\S 6.3.1.(a)]{bouthier2022perverse}).

\begin{proposition}\label{perversetstructures:prop:welldefined}
    The \( t \)-structure defined in \cref{perversetstructures:def:placidpresentation} is independent of the choice of perfectly placid presentation.

    Moreover, for a strongly pro-perfectly smooth morphism $f:X\to Y$ between algebraic spaces admitting a perfectly placid presentation, $f^!$ is $t$-exact.
\end{proposition}

\begin{proof}
    Let \( X \simeq \lim X_\alpha \) and \( X \simeq \lim X'_\beta \) be two perfectly placid presentations of \( X \). Write \( {}^p\indconstructible^{\geq 0}_1 \) and \( {}^p\indconstructible^{\geq 0}_2 \) for the resulting \( t \)-structures. It suffices to show that \( {}^p\indconstructible^{\geq 0}_1 = {}^p\indconstructible^{\geq 0}_2 \), since the negative truncation can be reconstructed from the nonnegative part: \( {}^p\indconstructible^{\leq 0} \) is the full $\infty$-subcategory of objects admitting no nontrivial morphisms to \( {}^p\indconstructible^{\geq 1} \).

    By \cref{dimthry:prop:uoequidimsatisfyIPP}, for every \( \alpha \) there exists \( \beta \) such that there is an equidimensional morphism \( X'_\beta \to X_\alpha \) under \( X \). Since pullback along equidimensional morphisms is left \( t \)-exact (as shown analogously to \cite[\S6.2.5(a)]{bouthier2022perverse}), we deduce that \( {}^p\indconstructible^{\geq 0}_1 \subset {}^p\indconstructible^{\geq 0}_2 \). By symmetry, we conclude that the two truncation full $\infty$-subcategories coincide.

    The proof of the second statement is analogous to \cite[Proposition 6.3.1(b)]{bouthier2022perverse}.
\end{proof}

\begin{para}\label{perversetstructures:def:placid}{\bf Perverse \( t \)-structures on perfectly placid \( \infty \)-stacks.}
    \begin{mylist}
        \item Let \( \varstack{X} \) be a perfectly placid \( \infty \)-stack. Recall from \cref{perfplacid:prop:canonicalsmoothcover} that \( \varstack{X} \simeq \colim_{X \to \varstack{X}} X \), where the colimit is taken over the $\infty$-category of perfectly smooth morphisms from affine schemes admitting perfectly placid presentations, and strongly pro-perfectly smooth morphisms between them.

        \item We then have \( \indconstructible(\varstack{X}) \simeq \lim^!_{X \to \varstack{X}} \indconstructible(X) \), and by \cref{perversetstructures:prop:welldefined}, the $\infty$-category  \( \indconstructible(\varstack{X})\) carries a unique \( t \)-structure 
        \(( {}^p\indconstructible^{\leq 0}(\varstack{X}),{}^p\indconstructible^{\geq 0}(\varstack{X}))\) such that
        \[
            K \in {}^p\indconstructible^{\leq 0}(\varstack{X}) \text{ (resp. }K \in {}^p\indconstructible^{\geq 0}(\varstack{X})) 
        \]
        if and only if for every perfectly smooth morphism \( f \colon X \to \varstack{X} \) from an algebraic space admitting a perfectly placid presentation, we have
        \[
            f^!K \in {}^p\indconstructible^{\leq 0}(X) \text{ (resp. } f^!K \in {}^p\indconstructible^{\geq 0}(X)).
        \]
    \end{mylist}
\end{para}

\begin{proof}
    This follows from the general fact that a limit in \( \prl \) of $\infty$-categories with \( t \)-structures and \( t \)-exact transition maps admits a canonical \( t \)-structure; see \cite[Lemma 6.1.3]{bouthier2022perverse}.
\end{proof}

\begin{para}\label{perversetstructures:def:placidlystratified}{\bf Perverse $t$-structures on perfectly placidly stratified $\infty$-stacks.}

\begin{mylist}
    \item Let $\varstack{Y}$ be a perfect $\infty$-stack, which admits a gluing of sheaves and is equipped with a perfectly placid stratification 
$\explicitset{\varstack{Y_{\alpha}}}_{\alpha\in I}$. Then for every collection of integers $p_{\nu}=\{\nu_{\alpha}\}_{\alpha\in I}$, which we call a {\em perversity function on $\varstack{Y}$}, there exists a unique $t$-structure $({}^{p_{\nu}}\indconstructible^{\leq 0}(\varstack{Y}), {}^{p_{\nu}}\indconstructible^{\geq 0}(\varstack{Y}))$ on $\indconstructible(\varstack{Y})$ such that
        \[
            {}^{p_{\nu}}\indconstructible^{\leq 0}(\varstack{Y}) = \explicitset{K\in\indconstructible(\varstack{Y})\mid\eta_{\alpha}^*K\in {}^{p}\indconstructible^{\leq -\nu_\alpha}(\varstack{Y}_{\alpha})}
        \]
         \[
            {}^{p_{\nu}}\indconstructible^{\geq 0}(\varstack{Y}) = \explicitset{K\in\indconstructible(\varstack{Y})\mid\eta_{\alpha}^!K\in {}^{p}\indconstructible^{\geq -\nu_\alpha}(\varstack{Y}_{\alpha})}.
        \]

        \item Let  $j \colon \varstack{U} \hookrightarrow \varstack{Y}$ 
be a pfp-open adapted embedding, let $p'_{\nu}=\{\nu_{\alpha}\}_{\alpha\in I_{\varstack{U}}}$ be the induced perversity function on $\varstack{U}$ 
and let $({}^{p'_{\nu}}\indconstructible^{\leq 0}(\varstack{U}), {}^{p'_{\nu}}\indconstructible^{\geq 0}(\varstack{U}))$  
be the corresponding perverse $t$-structure on $\indconstructible(\varstack{U})$. Then the pullback $j^!:\indconstructible(\varstack{Y})\to \indconstructible(\varstack{U})$ is $t$-exact, and for every $p'_{\nu}$-perverse sheaf $K\in \indconstructible(\varstack{U})$ there exists a unique extension $j_{!*}K\in \indconstructible(\varstack{Y})$ such that $j_{!*}K$ is $p_{\nu}$-perverse and $j_{!*}K$ has no non-zero subobjects and quotients supported on $\varstack{Y}\smallsetminus\varstack{U}$. Furthermore, for every $p'_{\nu}$-perverse sheaves $K,L\in \indconstructible(\varstack{U})$ the restriction map $j^!$ induces an isomorphism $\operatorname{Hom}_{\indconstructible(\varstack{Y})}(j_{!*}K,j_{!*}L)\overset{\sim}{\to} \operatorname{Hom}_{\indconstructible(\varstack{U})}(K,L)$.
\end{mylist}
\end{para}
\begin{proof}
Part (a)  follows as in \cite[Proposition 6.4.2]{bouthier2022perverse}, while part (b) follows as in \cite[Lemma~6.4.5 and Corollaries~6.4.9 and 6.4.10]{bouthier2022perverse}. 
\end{proof}


%



\begin{para} \label{perv:semismall}{\bf Application to semi-small morphisms.}
Assume that we are in the situation of \cref{stratification:def:semismall}, that is, $f:\varstack{X}\to\varstack{Y}$ be a morphism of perfect $\infty$-stacks, and $\explicitset{\varstack{Y_{\alpha}}}_{\alpha\in I}$ is a perfectly placid stratification such that $f$ is semi-small. In this case,  
 $\varstack{Y}$ has a natural perversity function $p_f=p_{\nu}$ such that $\nu(\alpha) = b_\alpha+\delta_\alpha$ for all $\alpha\in I$. 
 If, in addition, $\varstack{Y}$ admits gluing of sheaves, then $p_f$ gives rise to the perverse $t$-structure 
 $({}^{p_{f}}\indconstructible^{\leq 0}(\varstack{Y}), {}^{p_{f}}\indconstructible^{\geq 0}(\varstack{Y}))$ on $\indconstructible(\varstack{Y})$
(see \cref{perversetstructures:def:placidlystratified}(a)). 
\end{para}

\begin{theorem}\label{perversetstructures:prop:smallmapexactness}
In the situation of \cref{perv:semismall}, assume that $f$ is ind-pfp proper, \( \varstack{X} \) is perfectly smooth, and \( \varstack{Y} \) admits gluing of sheaves. Then:
    
    \begin{mylist}
        \item The pushforward \( K := f_!(\omega_{\varstack{X}}) \in\indconstructible(\varstack{Y}) \) is $p_f$-perverse.
        
        \item If \( j \colon \varstack{U} \hookrightarrow \varstack{Y} \) is an adapted pfp-open embedding such that \( f \) is \( \varstack{U} \)-small, then there is a canonical isomorphism
        \[
            K \simeq j_{!*} j^! K.
        \]
    \end{mylist}
\end{theorem}

\begin{proof}
    The proof is almost identical to the non-perfect case (see \cite[Theorem 6.5.3]{bouthier2022perverse}).

    It relies on the facts that \( (f_\alpha)_! \) exists, that \( f_\alpha^! \) is left \( t \)-exact, that \( (f_\alpha)_![-2\delta_\alpha] \) is right \( t \)-exact, and that \( \omega_{\varstack{X}_\alpha} \in {}^p\indconstructible^{\geq 0}(\varstack{X}_\alpha) \) (see \cite[Proposition 5.3.7]{bouthier2022perverse}, \cite[Lemma 6.3.5(b)]{bouthier2022perverse}, \cite[Lemma 6.3.6]{bouthier2022perverse}, and \cite[Lemma 6.3.5(a)]{bouthier2022perverse}, respectively). Each of these arguments carries over to the perfect setting, provided we use perfectly smooth base change (see \cref{contructibleonperfect:prop:perfsmoothperfproperbasechange}) and \( t \)-exactness of \( ! \)-pullback along perfectly smooth morphisms (see \cref{perversetstructures:def:texactnessforperfsmooth}).

    The existence of the functor \( \eta_\alpha^* \) and its properties follow from the non-perfect case (see \cite[Lemma 6.3.5(c)]{bouthier2022perverse}). Indeed, the generalization from perfect schemes to perfect \( \infty \)-stacks proceeds exactly as in the arguments throughout \cite[\S 5.4]{bouthier2022perverse}. For perfect schemes, the results for pfp-open embeddings are a special case of results for fp-open embeddings. For pfp-closed embeddings of perfect schemes, the relevant results hold because a pfp-closed embedding is the perfection of an fp-closed embedding, and hence follow from \cref{contructibleonperfect:prop:Dfactorsthroughperfection}.

    The only additional difference is that one must use the ind pfp-proper base change theorem, which is formulated and proven exactly as in \cite[Proposition 5.3.7]{bouthier2022perverse}, except with \cref{contructibleonperfect:prop:perfsmoothperfproperbasechange} replacing the classical proper base change.
\end{proof}

\section{The main theorem}\label{mainthm}

\begin{para}
    \begin{mylist}
        \item Recall the Witt affine Grothendieck--Springer fibration 
        \[
            \overline{\thespringermap} \colon [\springertotalspace/\loopspace(\thegroup)] \to [\compactelements/\loopspace(\thegroup)]
        \]
        of \cref{wittaffinrspringer:gneral:definitionoffib}.

        \item\label{mainthm:def:springer_sheaf} Define the \emph{Witt affine Grothendieck–-Springer sheaf} by
        \[
            \mathcal{S} := \overline{\thespringermap}_!(\omega_{[\springertotalspace/\loopspace(\thegroup)]}),
        \]
        and let \[ \mathcal{S}_\bullet, \mathcal{S}_{\topnilp}, \mathcal{S}_{\leq 0} \] denote its \( ! \)-restrictions to
        \(
            [\compactelements_\bullet/\loopspace(\thegroup)], \quad [\compactelements_{\topnilp}/\loopspace(\thegroup)], \quad [\compactelements_{\leq 0}/\loopspace(\thegroup)],
        \)
        respectively. By ind-pfp-proper base change (which is proved from \cref{contructibleonperfect:prop:perfsmoothperfproperbasechange} the same as \cite[Proposition 5.3.7]{bouthier2022perverse}), we have natural identifications:
        \[
            \mathcal{S}_\bullet \simeq (\overline{\thespringermap}_\bullet)_!(\omega_{[\springertotalspace_\bullet/\loopspace(\thegroup)]}), \quad
            \mathcal{S}_{\topnilp} \simeq (\overline{\thespringermap}_{\topnilp})_!(\omega_{[\springertotalspace_{\topnilp}/\loopspace(\thegroup)]}), \quad
            \mathcal{S}_{\leq 0} \simeq (\overline{\thespringermap}_{\leq 0})_!(\omega_{[\springertotalspace_{\leq 0}/\loopspace(\thegroup)]}).
        \]

        \item The \( \infty \)-stack \( [\compactelements_\bullet/\loopspace(\thegroup)] \) admits gluing of sheaves. Indeed, arguing as in 
        \cite[Lemma~7.1.2]{bouthier2022perverse} it follows from the fact that \[
            [\compactelements_\bullet/\loopspace(\thegroup)] \simeq \colim_m [\compactelements_{\leq m}/\loopspace(\thegroup)]
        \]
        and each \( [\compactelements_{\leq m}/\loopspace(\thegroup)] \) is a quotient of an ind-perfectly placid perfect ind-scheme (see \cref{wittaffinrspringer:gneral:compisperfectlyplacid}) by an ind-perfectly placid perfect group ind-scheme. 
        
        \item\label{mainthm:def:tstructbullet} We endow \( [\compactelements_\bullet/\loopspace(\thegroup)] \) with the perverse $t$-structure, 
        corresponding to the \( [\compactelements_{\leq 0}/\loopspace(\thegroup)] \)-small morphism \( \overline{\thespringermap}_\bullet \), following  \cref{perv:semismall}.

        \item\label{mainthm:def:tstructbullettopnilp} Likewise, we endow \( [\compactelements_{\bullet,\topnilp}/\loopspace(\thegroup)] \) with the perverse $t$-structure associated to the semi-small morphism \( \overline{\thespringermap}_{\bullet,\topnilp} \), following \cref{perv:semismall}. 
    \end{mylist}
\end{para}

\begin{theorem}\label{mainthm:mainthm}
    \begin{mylist}
        \item The sheaf \( \mathcal{S}_\bullet \) is perverse with respect to the $t$-structure described in \cref{mainthm:def:tstructbullet}. Moreover, 
        \[
            \mathcal{S}_\bullet \simeq j_{!*}\mathcal{S}_{\leq 0},
        \]
        where \( j \colon [\compactelements_{\leq 0}/\loopspace(\thegroup)] \hookrightarrow [\compactelements_{\bullet}/\loopspace(\thegroup)] \) is the open embedding.

        \item There is a natural isomorphism 
        \[
            \End(\mathcal{S}_\bullet) \simeq  \End(\mathcal{S}_{\leq 0})\simeq \qlbar[\theextendedweylgroup].
        \]

        \item The sheaf \( \mathcal{S}_{\bullet,\topnilp} \) is perverse with respect to the $t$-structure described in \cref{mainthm:def:tstructbullettopnilp}.
    \end{mylist}
\end{theorem}

\begin{proof}
    Parts (a) and (c) follow directly from \cref{perversetstructures:prop:smallmapexactness}, using the semi-smallness and smallness conditions established in \cref{semismallness:prop:mainthm} and \cref{semismallness:prop:maincor}.

    The first isomorphism of part (b) follows from part~(a) and the last observation of \cref{perversetstructures:def:placidlystratified}(b), while the second isomorphism follows as in \cite[Lemma 5.6.5]{bouthier2022perverse}. 
\end{proof}

\appendix
\section{Flatness of the Chevalley morphism}\label[Appendix]{flatnessofchevalley}

\subsection*{Introduction}

Let \( \thegroup \) be a split reductive group over a field $F$ of characteristic zero, with Lie algebra \( \thelie \). It is a classical result in geometric representation theory (\cite{kostant1963lie}) that the Chevalley morphism
\[
    \chi \colon \thelie \to \chevalley := \thelie /\!/\thegroup := \spec F[\thelie]^\thegroup
\]
is flat with regular singularities.

Given an affine scheme \( X \) defined over \( F \), the \( n \)-th jet space of \( X \), denoted \( L_n^{\mathrm{eq},+}X \), has \( R \)-points given by \( X(R[[t]]/t^n) \). A general theorem (see \cite[Theorem A.4]{mustata2001jet} for this particular case) asserts that flatness with rational singularities of \( \chi \) is equivalent to the flatness of its jet spaces:
\[
    L_n^{\mathrm{eq},+}\thelie \to L_n^{\mathrm{eq},+}(\chevalley).
\]

There exists a positive characteristic version of this result as well.

\begin{theorem}\label{thm:chevflatcharp}
    \cite[Theorem 7.4.2.]{bouthier2022perverse}.
    Assume $F = \resfield$ and suppose $\operatorname{char}\resfield>2h$, where $h$ is the Coxeter number of $\thelie$.
    The jet spaces of the Chevalley morphisms over a field of characteristic \( p \),
    \[
        L_n^{\mathrm{eq},+}(\chi) \colon L_n^{\mathrm{eq},+}(\thelie) \to L_n^{\mathrm{eq},+}(\chevalley), \quad \text{and} \quad
        L_n^{\mathrm{eq},+}(\theiwahorychevalleymap) \colon \liewahory^{\mathrm{eq}}_n \to L_n^{\mathrm{eq},+} (\chevalley)
    \]
    are flat.
\end{theorem}

In this appendix, we prove a Witt vector analogue of this result.

\begin{theorem}\label{thm:main}
    The morphisms
    \[
        \arcspace_n(\chi) \colon \arcspace_n(\thelie) \to \arcspace_n(\chevalley)
        \quad \text{and} \quad
        \arcspace_n(\theiwahorychevalleymap )\colon \liewahory_n \to \arcspace_n( \chevalley)
    \]
    introduced in \cref{wittloopandarc:def:arcofchi} and \cref{wittloopandarc:def:arcofnu} respectively, are flat.
\end{theorem}

\subsection*{The ring schemes \texorpdfstring{$\bnplustilde,\bplustilde$}{Bn+, B+}}

\begin{para}
    \begin{mylist}
        \item 
        Consider the morphisms $\xi_1,\xi_2:(\sltwo)_{\perf}\to \witt(-)[[t]]$ of functors $(\peraff)^{\op} \to \catofsets$, 
        sending an $R$-point of $(\sltwo)_{\perf}$, represented by a matrix \[g=\begin{pmatrix}a & b \\ c & d\end{pmatrix}\in \sltwo(R), \] to elements \( \xi_1(g):= [a]\varpi + [b]t \) and \( \xi_2(g):= [c]\varpi + [d]t \) in $\witt(R)[[t]]$.
        \item Let
        \[
            \bplustilde,\, \bnplustilde \colon (\peraff)^{\op} \to \catofsets
        \]
        be the functors that send $\spec R$ to the set of pairs consisting of a morphism \( \spec R \to \sltwo \), represented by a matrix \( g\in \sltwo(R) \), and an element of
        \[
            \bplustilde(R,g):=\witt(R)[[t]]/(\xi_1(g))\text{ and }
            \bnplustilde(R,g):=\witt(R)[[t]]/(\xi_1(g), \xi_2(g)^n), 
        \]
        respectively.
        \item The functors \( \bplustilde \) and \( \bnplustilde \) are sheaves of rings over \( (\speciallineargroup_2)_{\perf} \). Moreover, they are representable by (perfect) affine ring schemes over \( (\speciallineargroup_2)_{\perf}. \)

        \item These objects were introduced by Bando in \cite{bando2023derived} to construct an equivalence between the affine Hecke categories in equal and mixed characteristics.
    \end{mylist}
\end{para}

\begin{lemma}\label{lemma:computations} \cite[Lemmas 2.2 and 2.3]{bando2023derived}
    The quotient ring scheme \( \bplustilde / (\xi_2) \) over $(\sltwo)_{\perf}$ is isomorphic to \( \affinespace^1_{\sltwo, \perf} \). In addition, for any perfect $\resfield$-algebra $R$, and matrix $g\in \sltwo(R)$, \( \xi_2(g) \) is not a zero divisor in \( \bplustilde (R,g) \).
\end{lemma}

\begin{proof}
    A matrix  \[g= \begin{pmatrix} {a} & {b} \\ {c} & {d} \end{pmatrix}\in \sltwo(R)\] gives rise to the matrix 
    \[ [g]:=\begin{pmatrix} \canlift{a} & \canlift{b} \\ \canlift{c} & \canlift{d} \end{pmatrix}\in \operatorname{Mat}_2(\witt(R)[[t]]), \] 
    whose determinant is \([ad] - [bc]\). Since \([ad] - [bc]\) is congruent to \([ad - bc]=1\) modulo \( \varpi \), the matrix $[g]$ is invertible in \( \witt(R)[[t]] \). Hence,  the ideal \( (\xi_1(g), \xi_2(g))\subset \witt(R)[[t]] \) equals  \( (\varpi, t) \). Therefore, we have a natural isomorphism 
    \[
        \bplustilde(R,g) / \xi_2(g) \cong \witt(R)[[t]] / (\varpi, t) \cong R,
    \]
    which identifies \( \bplustilde / \xi_2 \) with \( \affinespace^1_{\sltwo, \perf} \) as claimed.

    For the second claim, 
    suppose \( x, y \in \witt(R)[[t]] \) satisfy
    \[
        x \xi_1(g) + y \xi_2(g) = 0.
    \]
    Writing this as a matrix equation 
    \[
        \begin{pmatrix} x & y \end{pmatrix}
        \begin{pmatrix} \canlift{a} & \canlift{b} \\ \canlift{c} & \canlift{d} \end{pmatrix}
        \begin{pmatrix} \varpi \\ t \end{pmatrix} = 0,
    \]
    we conclude that
    \[
        \begin{pmatrix} x & y \end{pmatrix}
        \begin{pmatrix} \canlift{a} & \canlift{b} \\ \canlift{c} & \canlift{d} \end{pmatrix}
        = z\begin{pmatrix} t & -\varpi \end{pmatrix}
    \]
    for some \( z \in \witt(R)[[t]] \). Multiplying by the matrix 

    \[ \begin{pmatrix} \canlift{d} & -\canlift{b} \\ -\canlift{c} & \canlift{a} \end{pmatrix} \]
    on the right gives
    \[
       (\canlift{ad}- \canlift{bc})\begin{pmatrix} x & y \end{pmatrix}
        = z\begin{pmatrix}\xi_2(g) & -\xi_1(g) \end{pmatrix}.
    \]
    Since $\canlift{ad}- \canlift{bc}\in\witt(R)[[t]]$ is invertible, we conclude that $y$ is divisible by $\xi_1(g)$,  so \( \xi_2(g) \) is not a zero divisor in \( \bplustilde(R,g) \).
\end{proof}

\begin{corollary}[\cite{bando2023derived}, Lemma 2.3]\label{uniqeexpressioninB}
    Any element \( f \in \bplustilde(R,g) \) has a unique expression as
    \[
        f = \sum_{i=0}^{\infty} [a_i]\xi_2(g)^i.
    \]
    In particular, \( \bnplustilde \) is isomorphic to the perfection of \( \affinespace^n_{\sltwo} \) as a scheme.
\end{corollary}

\begin{proof}
    \cref{lemma:computations} identifies \( \bplustilde(R,g) / \xi_2(g) \cong R \), and shows that \( \xi_2(g) \) is not a zero divisor. The usual division argument then yields the unique expansion. The claim about \( \bnplustilde \) follows.
\end{proof}

\begin{corollary}
    The presheaves \( \bnplustilde, \bplustilde \) define perfect affine ring schemes over $\sltwo$. Moreover, \( \bnplustilde \) is perfectly smooth perfectly finitely presented over $\spec\resfield$.
\end{corollary}


Let \( H \subset \sltwo \) be the perfection of the Borel subgroup scheme of lower triangular matrices.

\begin{lemma}
    For any \( g \in \sltwo(R) \) and \[ h = \begin{pmatrix} a & 0 \\ c & a^{-1} \end{pmatrix} \in H(R), \] we have equalities of ideals in \( \witt(R)[[t]] \):
    \[
        (\xi_1(hg)) = (\xi_1(g)), \quad (\xi_1(hg), \xi_2(hg)) = (\xi_1(g), \xi_2(g)).
    \]
\end{lemma}

\begin{proof}
    The first equality holds since \( \xi_1(hg) = [a]\xi_1(g) \). For the second, note that
    \[
        \xi_2(hg) = [c]\xi_1(g) + [a^{-1}]\xi_2(g) \equiv [a^{-1}]\xi_2(g) \pmod{\xi_1(hg)},
    \]
    which implies the ideals \( (\xi_1(hg), \xi_2(hg)) \) and \( (\xi_1(g), \xi_2(g)) \) coincide.
\end{proof}

\begin{corollary}
    The relative ring schemes \( \bnplustilde, \bplustilde \to (\sltwo)_{\perf} \) descend to relative ring schemes
    \[
        \bnplus, \bplus \to (H\backslash \sltwo)_{\perf} \cong \projectivespace^1_{\perf}.
    \]
\end{corollary}

Recall that the isomorphism \( H\backslash \speciallineargroup_2 \overset{\sim}{\to} \projectivespace^1 \) is given by
\[
    \begin{pmatrix}a & b \\ c & d\end{pmatrix} \mapsto [a : b].
\]

\begin{remark}
    After passing to the quotient by \( H \), the section \( \xi_2 \) is no longer well-defined.
\end{remark}

\begin{proposition}\label{prop:bniswittortrivial}
    Over the point $[1:0] \in \projectivespace^1_{\perf}$, the rings \( \bplus(R) \) and \( \bnplus(R) \) are isomorphic to \( R[[t]] \) and \( R[[t]]/t^n \), respectively. Over any point in the complement, they are isomorphic to the rings \( \witt(R) \) and \( \witt(R)/\varpi^n \), respectively.
\end{proposition}

\begin{proof}
    For the first statement, we compute
    \[
        \bnplus(R,[1:0]) = \bnplustilde\left(R, \begin{pmatrix}1 & 0 \\ 0 & 1\end{pmatrix}\right) 
        = \witt(R)[[t]]/(\varpi, t^n) \simeq R[[t]]/t^n.
    \]

    For the second statement, let \( [a:b] \in \projectivespace^1_{\perf}(R) \) with \( b \) invertible, and choose \( c,d \in R \) such that 
    \[
        g=\begin{pmatrix}a & b \\ c & d\end{pmatrix} \in \speciallineargroup_2(R).
    \]
    
    Since $b$ is invertible, the equality \( \xi_1(g) = [a]\varpi + [b]t \) implies that $t=[b^{-1}]\xi_1(g)-[ab^{-1}]\varpi$, so reducing modulo 
    \( \xi_1(g) \), we have \( t \equiv -[ab^{-1}]\varpi \), which implies an isomorphism 
    \[
    \witt(R)[[t]]/(\xi_1(g))\simeq \witt(R).
    \]
    Moreover, we have 
    \[
        \xi_2(g) = [c]\varpi + [d]t \equiv [c]\varpi - [d][ab^{-1}]\varpi = \left([c] - [d][ab^{-1}]\right)\varpi.
    \]
    By the determinant condition \( ad - bc = 1 \), we find that \( [c] - [d][ab^{-1}] \equiv [b^{-1}] \mod \varpi \), so it is invertible. Thus, \( \xi_2(g) \equiv u\varpi \) for some unit \( u \), so we have equalities of ideals \( (\xi_2(g)^n) = (\varpi^n) \) in $\witt(R)[[t]]/(\xi_1(g))\simeq \witt(R)$, giving the desired isomorphism.
\end{proof}

\begin{proposition}\label{prop:Bisaffinespace}
    As a scheme, $\bnplus$ is the perfection of an $n$-dimensional vector bundle over $\projectivespace^1_{\perf}$.
\end{proposition}

\begin{proof}
    Over the open perfect subscheme where $[a : -1]$ is a coordinate, we may use the section 
    \[
        \begin{pmatrix}a & -1 \\ 1 & 0 \end{pmatrix}
    \]
    for the map $\sltwo \to \projectivespace^1$. In this chart, we have $\xi_2 \equiv \varpi \mod \xi_1$ up to an invertible factor, so we obtain an isomorphism
    \[
        \affinespace^n \to \bnplus,\qquad (a_0,\dots,a_{n-1}) \mapsto \sum_{i=0}^{n-1} [a_i] \varpi^i.
    \]

    Over the open perfect subscheme where $[1 : b]$ is a coordinate, we use the section 
    \[
        \begin{pmatrix}1 & b \\ 0 & 1\end{pmatrix}.
    \]
    In this chart, $\xi_2 \equiv t \mod \xi_1$ up to an invertible factor, so we similarly get an isomorphism
    \[
        \affinespace^n \to \bnplus,\qquad (a_0,\dots,a_{n-1}) \mapsto \sum_{i=0}^{n-1} [a_i] t^i.
    \]

    On the overlap (where both $a$ and $b$ are invertible), we have the relation $\varpi = [-b/a] t$, which gives the transition function. Thus, the gluing is described by the linear transformation
    \[
        (a_0,\dots,a_{n-1}) \mapsto (a_0, (-b/a)a_1, \dots, (-b/a)^{n-1} a_{n-1}),
    \]
    which is $R$-linear in the coordinates $a_i$.
\end{proof}

\begin{remark}
    The vector bundle structure on $\bnplus$ does not coincide with its additive structure as a ring scheme outside the point $[1\!:\!0]$.
\end{remark}

\subsection*{Parametrized arc and jet spaces}

\begin{definition}\label{def:jetandarc}
    Let \( X \) be a scheme defined over \( \basicints \). The parametrized \( n \)-th jet space of \( X \), denoted \( \pararcspace_n (X) \), is defined as the functor sending  a perfect $\resfield$-algebra $R$ to the set of pairs consisting of a point $[a:b]\in \projectivespace_{\perf}^1(R)$ and an element in the set \[ \pararcspace_n (X)(R, [a \colon b]) := X(\bnplus(R, [a \colon b])).\] 
    
    The parametrized arc space \( \pararcspace(X) \) is defined as the functor sending  a perfect $\resfield$-algebra $R$ to the set of pairs consisting  of a point $[a:b]\in \projectivespace_{\perf}^1(R)$ and an element in the set  \[\pararcspace(X)(R, [a \colon b]) := X(\bplus(R, [a \colon b])).\]
\end{definition}

\begin{proposition}\label{prop:fibers_over_pararc_spaces}
    Let \( X \) be a scheme defined over \( \basicints \). The fibre of \( \pararcspace_n(X) \) over \( [1 \colon 0] \) is the perfection of the usual equal characteristic jet space \( \mathcal{L}^{\mathrm{eq},+}_n \colon R \mapsto X(R[[t]]/t^n) \). The fibre over other closed points of $\projectivespace^1_{\perf}$ is the Witt-vectors jet space \( \arcspace_n \colon R \mapsto X(\witt(R)/\varpi^n) \).
\end{proposition}

\begin{proof}
    This is an immediate consequence of \cref{prop:bniswittortrivial}.
\end{proof}

In the next section we will be particularly interested in the case, where \( X \) is isomorphic to \( \affinespace^k_{\basicints} \).

\begin{proposition}\label{prop:paramjetofaffineisvb}
    The parametrized jet space \( \pararcspace_n(\affinespace^k_{\basicints}) \) is isomorphic, as a scheme, to a vector bundle over \( \projectivespace^1_{\perf} \).
\end{proposition}

\begin{proof}
    This follows from \cref{prop:Bisaffinespace} and the canonical isomorphism \[ \affinespace^k(\bnplus(R,[a:b])) \cong \bnplus(R,[a:b])^k.\]
\end{proof}

\begin{proposition}\label{prop:gmequivariance}
    Let \( f \colon X \rightarrow Y \) be an equivariant morphism between schemes with an action of \( (\multiplicativegroup)_{\basicints} \). Then \( \pararcspace_n(X) \), \( \pararcspace_n(Y) \) carry a \( (\multiplicativegroup)_{\projectivespace^1_{\perf}} \)-action, and 
    \( \pararcspace_n(f) \colon \pararcspace_n(X) \rightarrow \pararcspace_n(Y) \) is a \( (\multiplicativegroup)_{\projectivespace^1_{\perf}} \)-equivariant morphism.
\end{proposition}
\begin{proof}
    The Teichmüller lift gives a morphism \( (\multiplicativegroup)_{\projectivespace^1_{\perf}} \rightarrow \pararcspace_n(\multiplicativegroup) \) of group schemes over $\projectivespace^1_{\perf}$,  sending a pair \( (\lambda, [a:b]) \in R^\times \times \projectivespace^1(R) \) to \( [\lambda] \in B_n^+(R,[a:b])^\times \). Since 
     \( \pararcspace_n(f) \colon \pararcspace_n(X) \rightarrow \pararcspace_n(Y) \) is an \(\pararcspace_n(\multiplicativegroup)\)-equivariant morphism, 
    we get a \( (\multiplicativegroup)_{\projectivespace^1_{\perf}} \)- action by restriction.
\end{proof}

\begin{example} \label{Ex:gmaction}
    Let \( X = \affinespace^k_{\basicints} \). The \( (\multiplicativegroup)_{\projectivespace^1_{\perf}} \)-action on \( \pararcspace_n(X) \), given by its structure as a vector bundle, agrees with the one coming from \cref{prop:gmequivariance}.
\end{example}
\begin{proof}
    Indeed, given \( (\lambda, [a:b]) \in (\multiplicativegroup)_{\projectivespace^1_{\perf}} \), the action of \( [\cdot] \colon \multiplicativegroup \rightarrow \pararcspace_n(\multiplicativegroup) \) is by multiplication of elements in \( \bnplus(R,[a:b])^k \) by \( [\lambda] \). Over the image of \( \spec R[1/a] \), this is \( \sum [a_i] t^i \mapsto \sum [\lambda a_i] t^i \), and over the image of \( \spec R[1/b] \), this is \( \sum [a_i] \varpi^i \mapsto \sum [\lambda a_i] \varpi^i \).
\end{proof}

\subsection*{The Chevalley morphism}\label{sec:Chev} 
Our  main \cref{thm:main} follows from the following result: 

\begin{theorem}\label{thm:mainact}
    Let \(\pararcspace_n(\chi) \colon \pararcspace_n(\thelie) \rightarrow \pararcspace_n(\chevalley)\) be the parametrized jet space morphism of \(\chi\) (see \cref{def:jetandarc} and \cref{wittloopandarc:def:chevalley}).  
    There exists an open neighbourhood $U$ of \([1\!:\!0] \in  \projectivespace^1_{\perf}\) such that \(\pararcspace_n(\chi)\mid_U \colon \pararcspace_n(\thelie)\mid_U \rightarrow \pararcspace_n(\chevalley)\mid_U\) is flat.
\end{theorem}


\begin{proof}[Proof of \cref{thm:main} assuming  \cref{thm:mainact}]
    We show here the case of $\arcspace_n(\thechevalleymap)$, while the case of $\arcspace_n(\theiwahorychevalleymap)$ (\cref{cor:lie_Iwahory_flat}) is a consequence of \cref{prop:generic_flta_in_lie_iwahory_case} using the same argument.
    
    Since $\resfield$ is algebraically closed, it is  infinite, so there exists a point $x\in U(\resfield)\smallsetminus 
    \{[1\!:\!0]\}$. 
    Then, by \cref{prop:fibers_over_pararc_spaces}, the fiber of $\pararcspace_n(\chi)$ over $x$ is isomorphic to the Witt-vector Chevalley morphism $\arcspace_n(\thechevalleymap):\arcspace_n(\thelie)\to \arcspace_n(\chevalley)$, and so it is flat as a base change of a flat morphism.
\end{proof}

Our next goal is to prove \cref{thm:mainact}.

\begin{lemma} \label{lem:vectorbundles}
    Both \(\pararcspace_n(\thelie)\) and \(\pararcspace_n(\chevalley)\) are perfections of vector bundles over \(\projectivespace^1\). Therefore, they are flat over \(\projectivespace^1_{\perf}\). 
\end{lemma}
\begin{proof}
    The first assertion follows from \cref{prop:paramjetofaffineisvb}. The second one follows from   \cref{perfschemes:prop:propertiesofperfschemespreserved}(d). 
\end{proof}

\begin{lemma}
    \(\pararcspace_n(\thelie)\), \(\pararcspace_n(\chevalley)\), and \(\projectivespace^1_{\perf}\) are perfectly finitely presented over \(\spec \resfield\).
\end{lemma}
\begin{proof}
Since the perfection of a scheme finitely presented over \(\spec \resfield\) is perfectly finitely presented by definition, the result follows from \cref{lem:vectorbundles}.
\end{proof}

\begin{proposition}
    There exists an open non-empty  perfect subscheme \( V \subset \pararcspace_n(\thelie) \) such that the restriction of \( \pararcspace_n(\chi) \) to $V$ is flat. Moreover, \( V \) contains the entire preimage of the point \([1\colon 0]\).
\end{proposition}
\begin{proof}
    The first statement follows from \cref{prop:flatlocusopen}, which guarantees that the flat locus of a morphism is open. 
    The second follows from applying \cref{thm:critflatfibre} to the fiber over \([1\colon 0]\), and noting that 
    by \cref{prop:fibers_over_pararc_spaces}, this fiber is isomorphic to the equidimensional  Chevalley morphism $L_n^{\mathrm{eq},+}(\thechevalleymap):L_n^{\mathrm{eq},+}(\thelie)\to L_n^{\mathrm{eq},+}(\chevalley)$, so it is flat by \cref{thm:chevflatcharp}.
\end{proof}

Now, we recall that \( \chi \colon \thelie \to \chevalley \) is \( \multiplicativegroup \)-equivariant with respect to the action on \( \thelie \) as a vector space. Thus:

\begin{proposition}
    The morphism \( \pararcspace_n(\chi) \colon \pararcspace_n(\thelie) \to \pararcspace_n(\chevalley) \) carries an action of \( (\multiplicativegroup)_{\projectivespace^1_{\perf}} \), which agrees on the target with the action coming from the structure of \( \pararcspace_n(\thelie) \) as a vector bundle.
\end{proposition}

\begin{proof}
    This follows from a combination of  \cref{prop:gmequivariance} and \cref{Ex:gmaction}.
\end{proof}

\begin{proof}[Proof of \cref{thm:mainact}]
    Let \( 0 \colon \projectivespace^1_{\perf} \to \pararcspace_n(\thelie) \) be the zero section. Set \( U_1 := 0^{-1}(V) \), where $V$ is the open perfect subscheme over which \( \pararcspace_n(\chi) \) is flat. Thus $U_1\subset \projectivespace^1_{\perf}$ is the open perfect subscheme over which \( \pararcspace_n(\chi) \) is flat at the zero section.

    Since \( \pararcspace_n(\chi) \) is \( (\multiplicativegroup)_{\projectivespace^1_{\perf}} \)-equivariant, the open perfect subscheme \( V \subset \pararcspace_n(\thelie) \) is a union of \( (\multiplicativegroup)_{\projectivespace^1_{\perf}} \)-orbits. Hence, it corresponds to an open perfect subscheme 
    \[V' \subset \projectivespace(\pararcspace_n(\thelie)):= [\pararcspace_n(\thelie)\smallsetminus 0 (\projectivespace^1_{\perf})]/(\multiplicativegroup)_{\projectivespace^1_{\perf}}.  \] 

    Let \( Z \subset \projectivespace^1_{\perf} \) denote the image of \( \projectivespace(\pararcspace_n(\thelie)) \smallsetminus V' \) under the  projection $\projectivespace(\pararcspace_n(\thelie))\to \projectivespace^1_{\perf}$. This is the set of points over which there exists a nonzero point where \( \pararcspace_n(\chi) \) is not flat. Since \( \projectivespace(\pararcspace_n(\thelie)) \) is perfectly proper (hence universally closed), it follows that \( Z \) is closed.

    Define \( U_2 := \projectivespace^1_{\perf} \smallsetminus Z \) and let \( U := U_1 \cap U_2 \). Then for any point in \( U \), the morphism \( \pararcspace_n(\chi) \) is flat everywhere on the fibre over that point. The set \( U \) is open (as an intersection of opens) and nonempty since \( [1\colon 0] \in U \).
\end{proof}

\subsection*{Flatness for \texorpdfstring{\( \liewahory \)}{the Lie algebra of the Iwahory}}


\begin{definition}
    The morphism of ring $\projectivespace^1_{\perf}$-schemes \( \bnplus \to \affinespace^1_{\projectivespace^1_{\perf}} \), given by modding out by the ideal \( (\xi_2) \), induces a morphism
    \[
        \pararcspace_n(\thelie) \to \pararcspace_0(\thelie).
    \]
    We define \( \paraliewahory_n \) to be the preimage of \( \pararcspace_0(\theborellie) \).
\end{definition}

\begin{proposition}
    As a scheme, \( \paraliewahory_n \) is the perfectization of a sub-vector bundle of \( \pararcspace_n(\thelie) \).
\end{proposition}
\begin{proof}
    Take a basis \( (x_1, \ldots, x_{\dim \theborellie}, \ldots, x_{\dim \thelie}) \) for \( \thelie \) over $\basicints$ such that \( (x_1, \ldots, x_{\dim \theborellie}) \) is a basis for \( \theborellie \). Then, for $[a:b] \in \projectivespace^1_{\perf}(R)$, we identify \( \pararcspace_n(\thelie) \cong \bigoplus \bnplus(R,[a:b]) \cdot x_i \).

    Recall the trivializations of \( \bnplus \), given in the proof of \cref{prop:Bisaffinespace}. 
    Outside the point \( [1\!:\!0] \), we can take a basis for \( \paraliewahory_n \) as
    \[
        \explicitset{ \varpi^i x_j \mid 0 < i < n \text{ if } j \leq \dim \theborellie;\;\; 0 \leq i < n \text{ otherwise} }.
    \]
    Outside \( [0\!:\!1] \), we can take a basis for \( \paraliewahory_n \) as
    \[
        \explicitset{ t^i x_j \mid 0 < i < n \text{ if } j \leq \dim \theborellie;\;\; 0 \leq i < n \text{ otherwise} }.
    \]
    
    One sees that these bases span the same subspace outside both \( [1\!:\!0] \) and \( [0\!:\!1] \).
\end{proof}

\begin{proposition}\label{prop:generic_flta_in_lie_iwahory_case}
    The schemes \( \paraliewahory_n \) and \( \pararcspace_n(\chevalley) \) are both perfections of vector bundles over \( \projectivespace^1 \). Therefore, they are flat over \( \projectivespace^1_{\perf} \), and \( \paraliewahory_n \), \( \pararcspace_n(\chevalley) \), and \( \projectivespace^1_{\perf} \) are perfectly finitely presented over \( \spec\resfield \). 

    There exists an open perfect subscheme \( V \subset \paraliewahory_n \) where \( \pararcspace_n(\theiwahorychevalleymap) \) is flat, and this open perfect subscheme contains the entire preimage of the point \( [1\!:\!0] \).
\end{proposition}
\begin{proof}
    All these assertions follow by exactly the same arguments as their counterparts in the case of $\thelie$.
\end{proof}

\begin{corollary}
    There exists a Zariski open perfect subscheme \( [1\!:\!0] \in U \subset \projectivespace^1_{\perf} \) such that \( \pararcspace_n(\theiwahorychevalleymap)\mid_U \colon \paraliewahory\mid_U \rightarrow \pararcspace_n(\chevalley)\mid_U \) is flat.
\end{corollary}

\begin{proof}
    The same argument as in the proof of \cref{thm:mainact} applies, using the second assertion of \cref{thm:chevflatcharp} instead of the first one. 
\end{proof}

\begin{corollary}\label{cor:lie_Iwahory_flat}
The Chevalley morphism 
\(\arcspace_n(\theiwahorychevalleymap )\colon \liewahory_n \to \arcspace_n( \chevalley)\) is flat. 
\end{corollary}
\begin{proof}
    See the argument after \cref{thm:mainact}.
\end{proof}

\section{Codimension of subschemes and valuations of differentials}\label[Appendix]{appendixcodim}

\begin{para}
    Let \( \sourcelattice, \targetlattice \) be finite free \( \basicints \)-modules of the same rank \( n \). We view \( \sourcelattice, \targetlattice \) as affine spaces over \( \spec \basicints \), but  remember their structure as group schemes.
\end{para}

\begin{para}\label{appendixcodim:valofdifferential}{\bf Valuation of the differential.}
    \begin{mylist}
        \item Let \( f \colon \sourcelattice \to \targetlattice \) be a morphism of schemes over \( \spec \basicints \).
        \item Given linear coordinates (that is, isomorphisms of group schemes) \( \sourcelattice \simeq \affinespace^n_\basicints \), \( \targetlattice \simeq \affinespace^n_\basicints \), we get for every \( x \in \sourcelattice(\basicints) \) a Jacobian matrix \( \differential_x(f) \in M_n(\basicints) \), given by the partial derivatives of \( f \) at \( x \) according to the coordinates.
        \item We view \( \differential_x(f) \) as a linear morphism \( \differential_x(f) \colon \sourcelattice \to \targetlattice \).
        \item The matrix \( \differential_x(f) \) clearly depends on the coordinates, but the valuation of its determinant \( d(f;x) := \valuation(\det(\differential_x(f))) \) 
        depends only on \( x \) and \( f \).
    \end{mylist}
\end{para}

The goal of this appendix is to show the following theorem: 

\begin{theorem}\label{appendixcodim:main}
    Let \( \sourcelattice, \targetlattice \) be finite free \( \basicints \)-modules of the same rank \( n \). Let \( f \colon \sourcelattice \to \targetlattice \) be a morphism of schemes over \( \spec \basicints \). Let \( X \subset\arcspace (\sourcelattice) \) and \( Y \subset\arcspace (\targetlattice) \) be pfp-locally closed subschemes of pure codimensions \( c_X \) and \( c_Y \), respectively, such that \( X = \arcspace (f)^{-1}(Y) \). 

    Suppose \(\arcspace (f)|_X \colon X \to Y \) is finite étale, and that there exists \( e \in \nats \) such that \( d(f; x) = e \) for all \( x \in X \). Then \( c_Y = c_X + e \).
\end{theorem}

\begin{para}\label{appendixcodim:setuplinearcase}{\bf Linear case.}
    \begin{mylist}
        \item Suppose \( f \colon \sourcelattice \to \targetlattice \) is an \( \basicints \)-linear morphism.
        \item We choose coordinates as in \cref{appendixcodim:valofdifferential} such that \( f \) is represented by a matrix \( D \).
    \end{mylist}
\end{para}

\begin{lemma}\label{appendixcodim:lin_case}
    Let \( f \) be as in \cref{appendixcodim:setuplinearcase}. Then it induces a perfectly finitely presented closed embedding \( \arcspace( f) \colon \arcspace(\sourcelattice )\to \arcspace(\targetlattice )\) of pure codimension \( e = \valuation(\det f) \).
\end{lemma}

\begin{proof}
    We can assume that $\sourcelattice=\targetlattice=\basicints^n$. By the theorem on elementrary divisors, $f$ decomposes as $f=x\circ g\circ y$ such that $x$ and $y$ are linear isomorphisms and  $g(x_1,\ldots,x_n)=(\varpi^{m_1}x_1,....,\varpi^{m_n}x_n)$. 
    
    It suffices to show the assertion in the case $f=g$.  In this case, the assertion is clear.
\end{proof}

\begin{para}\label{appendixcodim:morenotations}{\bf More notations.}
    \begin{mylist}
        \item The image \( \arcspace( f)(\arcspace (\sourcelattice)) \) is a perfectly finitely presented closed subscheme of \( \arcspace( \targetlattice )\) of codimension \( e \) (as shown in the lemma).
        \item Forand a subscheme \( x \in \arcspace(\sourcelattice)(\resfield) \)  \( H \subset \arcspace(\sourcelattice) \), we denote by \( x + H \) the translation (shift) of \( H \) by \( x \).
        \item For a scalar \( \lambda \in \basicints \), we denote by \( \lambda \cdot - \) the corresponding linear endomorphism (multiplication by \( \lambda \)).
        \item\label{codim_of_lin_subspace} We will be interested in subschemes of the form \( y + \arcspace( f)(\uniformizer^M \cdot \arcspace(\sourcelattice)) \subset \arcspace(\targetlattice )\), which is a pfp-closed subscheme of pure codimension \( Mn + e \) by \cref{appendixcodim:lin_case}.
    \end{mylist}
\end{para}

\subsection*{The general case.}  

The following assertion can be viewed as an arc space version of the inverse function theorem. 

\begin{theorem}\label{thm:arc_inverse_func}
    Let \( f \colon \sourcelattice \to \targetlattice \) be a morphism of schemes over \( \spec \basicints \). Let \( x \in \sourcelattice(\basicints) \), and \( e = d(f;x) \) (see \cref{appendixcodim:valofdifferential} for notation). Then for every $M>e$, the morphism \( f \) restricts to an isomorphism between the closed subschemes \( x + \uniformizer^M \arcspace(\sourcelattice )\overset{\sim}{\to} f(x) + \arcspace( D_x(f))(\uniformizer^M \arcspace(\sourcelattice)) \) (see \cref{appendixcodim:morenotations}).
\end{theorem}

\begin{proof}
     We can assume that $\sourcelattice=\targetlattice=\basicints^n$ and $x=f(x)=0$. 
     
     Using the elementary divisors theorem and arguing as in the proof of  \cref{appendixcodim:lin_case}, we can assume that $f(x_1,.....,x_n)=(\pi^{m_1}x_1+f'_1, ....., \pi^{m_n}x_1+f'_n)$, where $f'_1,.....,f'_n$ are sums of monomials of degree $\geq 2$. Observe that $m:=\max(m_1,..,m_n)\leq e$.

    We perform a change of variables $\uniformizer^{m}x'_i= x_i$. We obtain that 
    \[
    f(\uniformizer^{m}x'_1,....,\uniformizer^{m}x'_n)=(\pi^{m_1+m}f''_1,....\pi^{m_n+m}f''_n),
    \]
    where $f'':=(f''_1,....,f''_n)$ is \'etale at $0$. Thus, we reduce to the case when $f$ is \'etale at $0$. 
    
    Let $U\subset\Omega$ be the open neighborhood of $0$ such that $f|_U$ is \'etale, and let $U'=f(U)\subset\targetlattice$. Now the assertion follows from the Cartesian diagram (see \cref{wittloopandarc:properties:arpreserveetale}):
    \[\begin{tikzcd}
	{\arcspace( U)} & {\arcspace( U')} \\
	{\arcspace_0( U)} & {\arcspace_0( U')},
	\arrow[from=1-1, to=1-2]
	\arrow[from=1-1, to=2-1]
	\arrow[from=1-2, to=2-2]
	\arrow[from=2-1, to=2-2]
    \end{tikzcd}\]
    which induces an isomorphism on the fibres of the vertical maps.
\end{proof}

\begin{lemma}\label{appendix:has_p_adic_nbrd}
    Let \(X \subset \arcspace(\sourcelattice) \) is a pfp-locally closed perfect subscheme. Then for every \(x\in X\) there exists \( M \) such that \( x + \uniformizer^M \arcspace(\sourcelattice )\subset X \).
\end{lemma}

\begin{proof}
  Since  \(X \subset \arcspace(\sourcelattice) \) is pfp-locally closed, there exists \( M \) and a locally closed perfect subscheme \( X_M \subset \arcspace_M(\sourcelattice) \) such that \( X \simeq X_M \times_{\arcspace_M(\sourcelattice)} \arcspace(\sourcelattice )\). Let \( x_M \in X_M \) be the image of \( x \). Then the assertion follows from the Cartesian diagram:
    \[
    \begin{tikzcd}
        {x + \uniformizer^M \arcspace(\sourcelattice)} & X & {\arcspace(\sourcelattice)} \\
        {\explicitset{x_M}} & {X_M} & {\arcspace_M(\sourcelattice).}
        \arrow[from=1-1, to=1-2]
        \arrow[from=1-1, to=2-1]
        \arrow[from=1-2, to=1-3]
        \arrow[from=1-2, to=2-2]
        \arrow[from=1-3, to=2-3]
        \arrow[from=2-1, to=2-2]
        \arrow[from=2-2, to=2-3]
    \end{tikzcd}
    \]
\end{proof}

\begin{proof}[Proof of \cref{appendixcodim:main}]

    It follows from \cref{appendix:has_p_adic_nbrd} and  \cref{thm:arc_inverse_func} that there exists \( M \gg 0 \) such that \( x + \uniformizer^M \arcspace(\sourcelattice) \subset X \), \( y + D_x(f)(\uniformizer^M \arcspace(\sourcelattice)) \subset Y \), and \( f \) restricts to an isomorphism
    \[
    x + \uniformizer^M \arcspace(\sourcelattice) \xrightarrow{\sim} f(x) + \arcspace( D_x(f))(\uniformizer^M \arcspace(\sourcelattice)).
    \]

    Then, we get a commutative diagram:
    \[
    \begin{tikzcd}
        \arcspace(\sourcelattice) \arrow[r, "f"] & \arcspace(\targetlattice) \\
        X \arrow[u, hook, "\eta_X"] \arrow[r, "f|_X"] & Y \arrow[u, hook', "\eta_Y"] \\
        x + \uniformizer^M \arcspace(\sourcelattice) \arrow[u, hook, "\iota_X"] \arrow[r, "\sim"] & f(x) + \arcspace (D_xf)(\uniformizer^M \arcspace(\sourcelattice)). \arrow[u, hook', "\iota_Y"]
    \end{tikzcd}
    \]

    In the lower square, the two vertical morphisms are pfp-closed embeddings with sources that are connected and strongly pro-perfectly smooth. Thus, by \cref{dimtry:prop:uo_equidim_cover_by_smooth_imply_uo_equidim} and \cref{dimthry:prop:pfpfbetweenplacidpresentedcorollaries}, they are of pure codimension. The two lower horizontal morphisms are both weakly equidimensional of constant dimension 0: the middle horizontal morphism \( f|_X \) is étale, and the lowest horizontal morphism is an isomorphism. By additivity of codimension for weakly equidimensional morphisms of constant dimension (see \cref{dimthry:prop:compositionofequidim}), the codimensions of the vertical arrows $\iota_X$ and $\iota_Y$ are equal. We denote this codimension by \( d \).

    Applying \cref{dimthry:prop:compositionofequidim} again to the compositions \( \eta_X \circ \iota_X \) and \( \eta_Y \circ \iota_Y \) and using \cref{codim_of_lin_subspace}, we get
    \[
    Mn = c_X + d \quad \text{and} \quad Mn + e = c_Y + d.
    \]
    Subtracting these equations yields \( c_Y = c_X + e \), as claimed.
\end{proof}

\printbibliography

\end{document}